\documentclass[12pt]{amsart}
\usepackage[top=30truemm,bottom=30truemm,left=25truemm,right=25truemm]{geometry}
\usepackage{mathrsfs}

\usepackage{color}
\usepackage{bm}
\usepackage{amsfonts,amssymb}
\usepackage{dsfont}
\usepackage{amscd}
\usepackage{extarrows}
\usepackage{amsmath}
\usepackage{mathrsfs}
\usepackage{enumerate}
\usepackage{amscd}
\usepackage[all]{xy}
\usepackage[pagebackref,colorlinks]{hyperref}

\theoremstyle{plain} 
\newtheorem{theorem}{\indent\bf Theorem}[section]

\newtheorem{conjecture}[theorem]{\indent\bf Conjecture}
\theoremstyle{definition} 

\newtheorem{thm}{Theorem}[section]
\newtheorem{cor}[thm]{Corollary}
\newtheorem{lem}[thm]{Lemma}
\newtheorem{prop}[thm]{Proposition}

\theoremstyle{definition}
\newtheorem{defn}{Definition}[section]

\theoremstyle{remark}
\newtheorem{rem}{Remark}[section]

\newcommand{\be}{\begin{equation}}
	\newcommand{\ee}{\end{equation}}
\newcommand{\bea}{\begin{eqnarray}}
	\newcommand{\eea}{\end{eqnarray}}
\newcommand{\ben}{\begin{eqnarray*}}
	\newcommand{\een}{\end{eqnarray*}}
\newcommand{\bt}{\begin{split}}
	\newcommand{\et}{\end{split}}
\newcommand{\bet}{\begin{equation}}

\begin{document}
		\title[Degenerate complex Monge-Amp\`ere equation]{Degenerate complex Monge-Amp\` ere type equations on compact Hermitian manifolds and applications}
		
		%
	\author[Y. Li]{Yinji Li}
	\address{Yinji Li:  Institute of Mathematics\\Academy of Mathematics and Systems Sciences\\Chinese Academy of
		Sciences\\Beijing\\100190\\P. R. China}
	\email{1141287853@qq.com}
		\author[Z. Wang]{Zhiwei Wang}
		\address{Zhiwei Wang: Laboratory of Mathematics and Complex Systems (Ministry of Education)\\ School of Mathematical Sciences\\ Beijing Normal University\\ Beijing 100875\\ P. R. China}
		\email{zhiwei@bnu.edu.cn}
		\author[X. Zhou]{Xiangyu Zhou}
		\address{Xiangyu Zhou: Institute of Mathematics\\Academy of Mathematics and Systems Sciences\\and Hua Loo-Keng Key
			Laboratory of Mathematics\\Chinese Academy of
			Sciences\\Beijing\\100190\\P. R. China}
		\address{School of
			Mathematical Sciences, University of Chinese Academy of Sciences,
			Beijing 100049, P. R. China}
		\email{xyzhou@math.ac.cn}
		
		\begin{abstract}
We show  the existence and uniqueness of bounded solutions to the degenerate complex Monge-Amp\`ere type equations on compact Hermitian manifolds. We also study the asymptotics  of these solutions. As applications, we give
partial answers to the Tosatti-Weinkove conjecture and Demailly-P\u aun conjecture.

		\end{abstract}
		
		\thanks{}

		\maketitle
		\tableofcontents
	\section{Introduction}

In a celebrated paper \cite{Yau78} published in 1978, S.-T. Yau solved the Calabi conjecture. Since then, the complex Monge-Amp\`ere  (CMA for short) equation played a prominent role in complex geometry. Another breakthrough concerning the study of CMA equations was achieved by Bedford-Taylor \cite{BT76,BT82}. They initiated a new method for the study of very degenerate CMA equations. 
Combining these results, Kolodziej \cite{Kol98} proved the existence of solutions for equations of type 
$$(\omega+dd^c\varphi)^n=v$$
on a compact K\"ahler manifold $X$, where $\omega$ is a K\"ahler metric and $v\geq 0$ a (degenerate) density in $L^p$ for some $p>1$.  
In various geometric problems, it is necessary to consider the case where $\omega$ is merely semi-positive. This more difficult situation has been studied first by Tsuji \cite{Tsu88}, and been extensively studied recently, see e.g.  \cite{Kol03,TZ06,Zha06,Pau08,EGZ09,DP10,BEGZ10,EGZ11,EGZ17}. For more related   applications, see e.g.  \cite{Sib99,DP04,TWY15}.

The Hermitian version of Calabi conjecture has also been studied in \cite{Che87,Han96,GL10},  and ultimately in  \cite{TW10}, Tosatti-Weinkove solved the equation in full generality, with   uniform $\mathscr C^\infty$ a priori estimates on  the solution. Recently,   there are many progress on the study of CMA equations with degenerate $L^p$ density for some  $p>1$ on compact Hermitian manifolds, see e.g.  \cite{Kol05,KN15,Ngu16}.

In this paper, we push further the techniques developed so far, and obtain very general results of degenerate CMA equations on compact Hermitian manifolds.

Let (X, $\omega$) be a compact Hermitian manifold of complex  dimension $n$, with a Hermitian metric $\omega$.
Let $\beta$ be a smooth real closed $(1,1)$ form. 
A  function $u:X\rightarrow [-\infty,+\infty)$ is called quasi-plurisubharmonic, if locally $u$ can be written as the sum of a smooth function and a psh function.  A  $\beta$-plurisubharmonic ($\beta$-psh for short) is a quasi-plurisubharmonic function $u$ such that    $\beta+dd^cu\geq 0$ in the sense of currents. 
The set of  all $\beta$-psh  functions on $X$ is denoted by $\mbox{PSH}(X,\beta)$. Suppose that  there exist a function  $\rho\geq 0 \in \mbox{PSH}(X,\beta)\cap L^{\infty}(X)$. We define the $(\beta+dd^c\rho)$-psh function to be a function $v$ such that $\rho+v\in \mbox{PSH}(X,\beta)$, and denote by $\mbox{PSH}(X,\beta+dd^c\rho)$ the set of all  $(\beta+dd^c\rho)$-psh functions. 

Given a non-negative function $f\in L^p(X,\omega^n)$, $p>1$, such that $\int_Xf\omega^n=\int_X\beta^n>0$, we study  the degenerate CMA  type equations
\begin{align}\label{equ: CMA}(\beta+dd^c\varphi)^n=e^{\lambda\varphi}f\omega^n, \lambda\geq 0, \varphi\in \mbox{PSH}(X,\beta).\end{align}
Here $d=\partial+\bar\partial$, $d^c=\frac{i}{2\pi}(\bar\partial-\partial)$, and $dd^c=\frac{i}{\pi}\partial\bar\partial$.


We first solve the   equation (\ref{equ: CMA})   for the case $\lambda=0$ by establishing the following
\begin{thm}\label{thm: main 1}Let $(X,\omega)$ be a compact Hermitian manifold  of complex dimension $n$.  Let $\beta$ be a smooth real $(1,1)$-form on $X$ such that there exists  $\rho\in \mbox{PSH}(X,\beta)\cap L^\infty(X)$. Let $0\leq f\in L^p(X,\omega^n)$, $p>1$, be such that $\int_Xf\omega^n=\int_X\beta^n>0$. Then there exists a unique  real-valued function $\varphi\in \mbox{PSH}(X,\beta )\cap L^\infty(X)$, satisfying
	\[(\beta+dd^c\varphi)^n=f\omega^n\]
	in the weak sense of currents, and $\|\varphi\|_{L^\infty(X)}\leq C(X,\omega,\beta, M,\|f\|_p)$.
\end{thm}
\begin{rem} When $(X,\omega)$ is compact  K\"ahler, $\beta$ is a  closed smooth semi-positive $(1,1)$-form,  and the density is $f\beta^n$, with $\int_Xf\beta^n=\int_X\beta^n>0$,  the same result was obtained by  Eyssidieux-Guedj-Zeriahi \cite{EGZ09}. Moreover, it is proved in \cite{EGZ09} that the solution  $\varphi$ can be continuous if $(X,\{\beta\})$ satisfies the continuous approximation property, i.e. for every $\beta$-psh function $\varphi$, there is a decreasing sequece of continuous $\beta$-psh functions converging to $\varphi$ pointwise.
	When $(X,\omega)$ is compact K\"ahler, $\beta$ is a closed smooth semi-positive $(1,1)$-form, and the density $v$ is any non-negative $L\log^{n+\varepsilon}L$-density such that $\int_Xv=\int_X\beta^n>0$, the same result was obtained by Demailly-Pali \cite{DP10}.
	When $(X,\omega)$ is compact Hermitian, and $\beta$ is a closed smooth semi-positive $(1,1)$-form,  the same result was obtained by Nguyen \cite{Ngu16}, by combining the method of Kolodziej \cite{Kol98,Kol05} and  Eyssidieux-Guedj-Zeriahi \cite{EGZ09,EGZ11}.

Theorem \ref{thm: main 1} in K\"ahler setting can be applied to construct various  singular K\"ahler-Einstein metrics on K\"ahler spaces, see  e.g. \cite{TZ06,Zha06,EGZ09,BEGZ10,DP10}.
	\end{rem}

The proof of Theorem \ref{thm: main 1} is much involved. We first use  Demailly's celebrated regularization theorem to get a decreasing sequence of smooth $\rho_j$, such that $\beta_j:=\beta+dd^c\rho_j+\frac{1}{j}\omega\geq \frac{1}{2j^2}\omega$, then solve (by Kolodziej-Nguyen's theorem  \cite{KN15}) the CMA equations $(\beta_j+dd^cu_j)^n=c_jf\omega^n$, with $\beta_j+dd^cu_j\geq 0$ in the weak sense of currents, with  $\sup_Xu_j=0$. By exploring a series of capacity estimates, energy estimates, we get a solution in the non-pluripolar sense (introduced by Boucksom-Eyssidieux-Guedj-Zeriahi  \cite{BEGZ10}). Then  we  derive  $L^\infty$-esimates  (inspired by Eyssidieux-Guedj-Zeriahi \cite{EGZ09}) to the non-pluripolar solution, thus prove that $\rho_j+u_j$ (upto choosing a subsequence) converges to the desired solution $\varphi$. Finally, we (following Dinew \cite{Din09-2} and Boucksom-Eyssidieux-Guedj-Zeriahi  \cite{BEGZ10}) use the comparison principle to prove the uniqueness of the solution.
	
\begin{rem} If the class $\{\beta\}$ is only assumed to be pseudo-effective (without the assumption of the existence of bounded potential $\rho$), then we can not define the non-pluripolar product for closed positive $(1,1)$-currents in the class $\{\beta\}$, and the method in this paper does not work.
\end{rem}

Second, we solve the equation (\ref{equ: CMA}) for  the case $\lambda>0$ by establishing the following
\begin{thm}\label{thm: main 2}
	Let $(X,\omega)$ be a compact Hermitian manifold  of complex dimension $n$.  Let $\beta$ be a smooth real $(1,1)$-form on $X$ such that there exists  $\rho\in \mbox{PSH}(X,\beta)\cap L^\infty(X)$, and $\int_X\beta^n>0$. Let $0\leq f\in L^p(X,\omega^n)$, $p>1$, be such that $\int_Xf\omega^n>0$. Let $\lambda>0$ be a positive number.  Then there is a unique real-valued function $\varphi\in \mbox{PSH}(X,\beta )\cap L^\infty(X)$, satisfying
	\[(\beta+dd^c\varphi)^n=e^{\lambda\varphi} f\omega^n\]
	in the weak sense of currents, and $\|\varphi\|_{L^\infty(X)}\leq C(X,\omega,\beta, M,\|f\|_p)$.
\end{thm}

\begin{rem}
	When $(X,\omega)$ is compact K\"ahler, and $\beta$ is a closed semi-positive $(1,1)$-form, the same result was obtained by Eyssidieux-Guedj-Zeriahi \cite{EGZ11,EGZ17}, using the method of viscosity solutions. 	
	When $(X,\omega)$ is compact Hermitian, and $\beta$ is a closed smooth semi-positive $(1,1)$-form,  the same result was obtained by Nguyen \cite{Ngu16}.
	\end{rem}

\begin{rem}
	In \cite{EGZ11,EGZ17}, it is proved that in the K\"ahler setting, the solution $\varphi$ in the above theorem is continuous on the ample locus of the class $\{\beta\}$, and the continuity at the boundary is left as an open question. In the Hermitian setting, one can not expect too much on the  positivities on the class $\beta$, the fact that we only get the boundedness of the solution seems to be reasonable.
	\end{rem}

The proof of Theorem \ref{thm: main 2} is inspired by   \cite{BEGZ10}. The main ingredient is the  Schauder's fixed point theorem and a comparison principle.

We also consider the evolution of the solutions for the degenerate CMA equations (\ref{equ: CMA}).

\begin{thm}\label{thm: main-3}Let $(X,\omega)$ be a compact Hermitian manifold of complex dimension $n$. Let $\{\beta\}\in H^{1,1}(X,\mathbb R)$ be a real $(1,1)$-class with smooth representative $\beta$. Assume that there is a bounded $\beta$-PSH function $\rho$, and $\int_X\beta^n>0$. Let $f\geq 0$, $f\in L^p(X,\omega^n)$, $p>1$,  such that $\int_Xf\omega^n>0$. Let $\lambda\geq 0$, and $\varphi_\lambda$ be the unique  weak solution to the CMA equation 
	\[	(\beta+dd^c \varphi_{\lambda})^n=e^{\lambda \varphi_{\lambda}+M_{\lambda}}f\omega^n,~ \varphi_\lambda\in \mbox{PSH}(X,\beta),~\sup_X \varphi_{\lambda}=0.	\]
	If for some $\lambda_0\geq 0$,  $\|\varphi_\lambda-\varphi_{\lambda_0}\|_1\rightarrow 0$ as $\lambda\rightarrow \lambda_0$, then 
	\[M_\lambda\rightarrow M_{\lambda_0}, \mbox{ and } \varphi_\lambda\rightarrow \varphi_{\lambda_0} \mbox{ uniformly on } X \mbox{ as } \lambda\rightarrow \lambda_0.\]
\end{thm}

The above solution to degenerate CMA equations  can be used to study   the following two conjectures, say the Tosatti-Weinkove conjecture and the Demailly-P\u aun conjecture.

\begin{conjecture}[The Tosatti-Weinkove conjecture]\label{conj: tw conj} 
	Let 	X be a compact n-dimensional complex manifold. Suppose $\{\beta\} \in H^{1,1}(X,\mathbb{R})$, which is nef and satisfies $\int_X\beta^n>0$. Fix  $x_1,...,x_N \in X$   and choose  positive real numbers  $\tau_1,...,\tau_N$ so that 
	$$\sum_{i=1}^{N}\tau_i^n <\int_X \beta^n.$$
	Then there exists a $\beta$-\mbox{PSH} function $\varphi$ with logarithmic poles at $x_1,...,x_N$:
	$$\varphi(z)\leq \tau_j log|z|+O(1),$$
	in a coordinate neighbourhood $(z_1,...,z_n)$ centered at $x_i$, where $|z|^2=|z_1|^2+...+|z_n|^2.$	
\end{conjecture}
\begin{rem}
	The K\"ahler version is proved by Demailly-P\u aun by using Demailly's celebrated mass concentration technique in \cite{Dem93}. The conjecture for $n=2$ and $n=3$ and some partial results for general $n$ (if $X$ is Moishezon and $\{\beta\}$ is a rational class) was obtained by Tosatti-Weinkove. Nguyen \cite{Ngu16} proved the Tosatti-Weinkove conjecture under the assumption that $\beta\geq0$, i.e., $\beta$ is a smooth semi-positive $(1,1)$-form. 
\end{rem}

\begin{conjecture}[{\cite[Conjeture 0.8]{DP04}}]\label{conj: dp conj}
	Let $(X,\omega)$ be a compact Hermitian manifold of complex dimension $n$. Let $\{\beta\}\in H^{1,1}(X,\mathbb R)$ be a real $(1,1)$-class with smooth representative $\beta$. Suppose $\{\beta\}$ is nef and $\int_X\beta^n>0$, then the class  $\{\beta\}$  is big, i.e., it contains a K\"ahler current.
\end{conjecture}  

\begin{rem}
	When $n=2$, and $\beta$ is semi-postive, Conjecture \ref{conj: dp conj} is true  by the work of N. Buchdahl \cite{Buc99, Buc00} and Lamari \cite{Lam1,Lam2}. When $n=3$,  Chiose \cite{Ch16} proved the Conjecture \ref{conj: dp conj} under the assumption that there exists a pluriclosed Hermitian metric. For  $n\geq 4$,   Nguyen \cite{Ngu16} proved the Conjecture \ref{conj: dp conj} is true under the assumption that $\beta$ is semi-positive and there is a pluriclosed Hermitian metric. Actually, all the previous results mentioned here  can also be seen as   partial answers to a conjecture of Boucksom \cite{Bou}, which asks for the existence of a K\"ahler currents in a pseudo-effective real $(1,1)$-class $\{\beta\}\in H^{1,1}(X,\mathbb R)$ on a compact Hermitian manifold $(X,\omega)$, under the assumption that $vol(\{\beta\})>0$ (for definition, see \cite{Bou}). Boucksom \cite{Bou} proved his conjecture when $(X,\omega)$ is assume to be compact K\"ahler. For compact Hermitian manifolds and general pseudo-effective class, the second author has partial results \cite{Wan16,Wan19}.
	
\end{rem}

As just mentioned, Nguyen proved the Conjecture \ref{conj: tw conj} under the assumption that the real $(1,1)$-class $\{\beta\}$ is semi-positive, and proved the Conjecture \ref{conj: dp conj} under the assumption that $
\beta$ is semi-positive and there is a pluriclosed Hermitian metric on $X$. The main purpose of this paper, is to relax the semi-positivity assumption of the class $\{\beta\}$, to the assumption that there is a bounded quasi-plurisubharmonic function $\rho$, such that $\beta+dd^c\rho\geq 0$ in the weak sense of currents. To this purpose, we need to solve the  degenerate CMA equations (\ref{equ: CMA}).

Our first application is to give a partial answer to the Tosatti-Weinkove conjecture (Conjecture \ref{conj: tw conj}).

\begin{thm}\label{thm: main4}	
	Let 	X be a compact n-dimensional complex manifold. Suppose $\{\beta\} \in H^{1,1}(X,\mathbb{R})$, $\beta+dd^c \rho \geq 0$ for bounded $\beta$-\mbox{PSH} function $\rho$,  $\int_X \beta^n \textgreater0$. Let $x_1,...,x_N \in X$ be fixed points and let $\tau_1,...,\tau_N$ be positive real numbers so that 
	$$\sum_{i=1}^{N}\tau_i^n < \int_X \beta^n.$$
	Then there exists a $\beta$-\mbox{PSH} function $\varphi$ with logarithmic poles at $x_1,...,x_N$:
	$$\varphi(z)\leq \tau_j \log|z|+O(1),$$
	in a coordinate neighbourhood $(z_1,...,z_n)$ centered at $x_i$, where $|z|^2=|z_1|^2+...+|z_n|^2.$	
\end{thm}
Our second application is to give a partial answer to the Demailly-P\u aun conjecture (Conjecture \ref{conj: dp conj}).
\begin{thm}\label{thm: main 5}
	Let 	$(X,\omega)$ is a compact Hermitian manifold, with  $\omega$ a  pluriclosed Hermitian metric, i.e. $dd^c\omega=0$. Let  $\{\beta\} \in H^{1,1}(X,\mathbb{R})$ be a real $(1,1)$-class with smooth representative $\beta$, such that  $\beta+dd^c \rho \geq 0$ for some bounded $\beta$-psh function $\rho$, and $\int_X\beta^n>0$. Then $\{\beta\}$ contains a K\"ahler current.
\end{thm}

\begin{rem}
	By the  Demailly's  regularization theorem, we know that the class $\beta$  with bounded $\beta$-psh potential  is nef,  thus Theorem \ref{thm: main4} confirms the   Conjecture \ref{conj: dp conj}, and Theorem \ref{thm: main 5}  confirms the Conjecture \ref{conj: dp conj}, under the assumption that  there is a bounded $\beta$-psh function $\rho$, and the Hermitian metric $\omega$ is pluriclosed.
\end{rem}

\subsection*{Acknowledgements} The second author would like to thank Ngoc Cuong Nguyen for many helpful discussions on the degenerate CMA equations.
This research is supported by National Key R\&D Program of China (No. 2021YFA1002600).  The authors are partially supported respectively by NSFC grants (12071035, 11688101).

\section{Preliminaries}

Let (X, $\omega$) be a compact Hermitian manifold of complex  dimension $n$, with a Hermitian metric $\omega$.
Let $\beta$ be a smooth real closed $(1,1)$ form. 
%
Suppose that  there exist a function  $0\leq \rho\in \mbox{PSH}(X,\beta)\cap L^{\infty}(X)$.

By Demailly's celebrated regularization theorem (see \cite{Dem12}), there exists a smooth decreasing sequence quasi-psh functions $\{\rho_j\}$ satisfying
\[\beta_j:=\beta +dd^c \rho_j +\frac{1}{j} \omega \geq  \frac{1}{2j^2}\omega,~~~\lim \rho_j=\rho.\]
It is easy to see that the sequence $\{\rho_j\}$ is uniformly bounded, and we fix  throughout the constant $$M:= \sup_j\|\rho_j\|_{L^\infty(X)}+\|\rho\|_{L^\infty(X)}.$$

We shall denote throughout, by $B$ the   "curvature" constant  $B >0$ satisfying
\begin{align}
	-B\omega^2 \leq dd^c \omega \leq B\omega^2, -B\omega^3\leq d\omega \wedge d^c \omega \leq B\omega^3.\label{equ:B cond}
	\end{align}

Set  $B_j=8Bj^4$. Direct computation implies that 
\[-B_j\beta_j^2 \leq dd^c \beta_j \leq B_j\beta_j^2, ~~~-B_j\beta_j^3 \leq d\beta_j \wedge d^c \beta_j \leq B_j\beta_j^3.\]
\subsection{Chern-Levine-Nirenberg inequalities}

\begin{prop}[CLN inequality]\label{prop:CLN ineq}
Let $u\in \mbox{PSH}(X,\beta_j)$ and $v_1,\cdots,v_n \in \mbox{PSH}(X,\beta_j)\cap L^{\infty}(X)$. If $u\leq -1$, then there exists a uniform constant $C=C(X,\omega,\beta,M)$ depending only on $(X,\omega,\beta,M)$, such that:
\begin{align*}\int_X-u(\beta_j+dd^cv_1)\wedge\cdots\wedge(\beta_j+dd^cv_n)\leq C||u||_{L^1(X)}(\|v_1\|_{L^\infty(X)}+1)\cdots(\|v_n||_{L^\infty(X)}+1).\end{align*}
\end{prop}
\begin{proof} The proof is essentially the same as \cite[Proposition 3.3, Remark 3.4]{Dem} up to minor modifications. 
For the sake of integrity, we include the proof here.
	
We cover X by finitely many  coordinate balls $\{U_i \subseteq \subseteq U_i'\}$, i=1,...,M. 
Denote by  $z_i=(z_{i,1},\cdots,z_{i,n})$ the  coordinates on $U_i'$, and  $|z_{i}|^2=|z_{i,1}|^2+\cdots+|z_{i,n}|^2$. 
Let $N=N(X,\omega,\beta)>0$ be a large constant such that $\beta_j=\beta+\frac{1}{j}\omega+dd^c\rho_j \leq Ndd^c|z_i|^2+dd^c\rho_j$ on each $U_i'$. Then 
\begin{align*}
&\int_{U_i}-u(\beta_j+dd^cv_1)\wedge\cdots\wedge(\beta_j+dd^cv_n)\\
&\leq \int_{U_i}-u(Ndd^c|z_i|^2+dd^c\rho_j+dd^cv_1)\wedge\cdots\wedge(Ndd^c|z_i|^2+dd^c\rho_j+dd^cv_n)\\
&=\int_{U_i}-(u+\rho_j+N|z_i|^2)(Ndd^c|z_i|^2+dd^c\rho_j+dd^cv_1)\wedge\cdots\wedge(Ndd^c|z_i|^2+dd^c\rho_j+dd^cv_n)\\
&+\int_{U_i}(\rho_j+N|z_i|^2)(Ndd^c|z_i|^2+dd^c\rho_j+dd^cv_1)\wedge\cdots\wedge(Ndd^c|z_i|^2+dd^c\rho_j+dd^cv_n).
	\end{align*}
By \cite[Proposition 3.3, Remark 3.4]{Dem}, the  first term is bounded by 
	$$C||u+\rho_j+N|z_i|^2||_{L^1(U_i')}||v_1+\rho_j+N|z_i|^2||_{L^{\infty}(U_i')}\cdots||v_n+\rho_j+N|z_i|^2||_{L^{\infty}(U_i')},$$ 
 and the second term is bounded  by 
\begin{align*}
C||v_1+\rho_j+N|z_i|^2||_{L^{\infty}(U_i')}\cdots|v_n+\rho_j+N|z_i|^2||_{L^\infty(U_i')},
\end{align*}
where $C$ is a uniform constant depending only on $(U_j, U_j')$. 

Since $\{\rho_j\}$ is uniformly bounded and $u\leq -1$,
\begin{align*}
&\int_X-u(\beta_j+dd^cv_1)\wedge...\wedge(\beta_j+dd^cv_n)\\
&\leq \sum_{i=1}^{M}\int_{U_i}-u(\beta_j+dd^cv_1)\wedge\cdots\wedge(\beta_j+dd^cv_n)\\
&\leq \sum_{i=1}^{M}C||u+\rho_j+N|z_i|^2||_{L^1(U_i')}||v_1+\rho_j+N|z_i|^2||_{L^{\infty}(U_i')}\cdots||v_n+\rho_j+N|z_i|^2||_{L^{\infty}(U_i')}\\
&\leq C(X,\omega,\beta,M)||u||_{L^1(X)}(\|v_1\|_{L^\infty(X)}+1)\cdots(\|v_n||_{L^\infty(X)}+1).
	\end{align*}
The proof of the Proposition \ref{prop:CLN ineq} is complete.
\end{proof}
\begin{rem}\label{rem: v<-1 cln in}
	If $v_j\leq -1$ for $j=1,\cdots,n$, it is easy to see from the proof that, the above estimate can be modified to 
	\begin{align*}\int_X-u(\beta_j+dd^cv_1)\wedge\cdots\wedge(\beta_j+dd^cv_n)\leq C||u||_{L^1(X)}\|v_1\|_{L^\infty(X)}\cdots\|v_n||_{L^\infty(X)}.\end{align*}
\end{rem}
\begin{prop}\label{prop:uniform est}
	There exist constants $\alpha,C\textgreater 0$, depending only on  $(X,\omega,\beta, M)$, such that for any $v\in \mbox{PSH}(X,\beta+dd^c \rho_j), \sup_Xv=0,$
	$$\int_X -v \omega^n \leq C, \int_X e^{-\alpha v} \omega^n \leq C.$$
\end{prop}
\begin{proof}
Since  $\{\rho_j\}$ is uniformly bounded, the lemma follows from  {\cite[Lemma 2.3]{DP10}}.
	\end{proof}

\subsection{Extremal functions}

For a Borel subset  $K $  of $X$,  the extremal function $V_{K,j}$  of $K$ with respect to $\beta+dd^c\rho_j$ is defined as 
\begin{center}
	$V_{K,j}:=\sup\{v\in\mbox{PSH}(X,\beta+dd^c\rho_j):v\leq0$ on $ K \}$.
\end{center}

Denote by $V_{K,j}^*$   the upper semi-continuous regularization of $V_{K,j}$. We say $K$ is a $(\beta+dd^c\rho_j)$-pluripolar set, if there exists $u\in \mbox{PSH}(X,\beta+dd^c\rho_j)$, such that $K\subset \{u=-\infty\}$.

\begin{lem}\label{lem: non pluripolar char}
	$K$ is a $(\beta+dd^c\rho_j)$-pluripolar set
if  and only if  $\sup_XV_{K,j}=+\infty$.
	
\end{lem}
\begin{proof} The proof follows \cite[Theorem 5.2]{GZ05}. By considering $u+k$ for $k=1,2,\cdots$, one can see that $\sup_X V_{K,j}=+\infty$.
Conversely, let $v_k\in \mbox{PSH}(X,\beta+dd^c\rho_j)$ be a sequence such that $N_k:=\sup_Xv_k\geq2^k$ and $v_k\leq 0$ on K. Since  $\sup_X(v_k-N_k)=0$,  the sequence $\{v_k-N_k\}$ is compact in $L^1(X,\omega^n)$. 
Then the function   
$$u:=\sum_{k=1}^{\infty}\frac{v_k-N_k}{2^k}$$
as a  decreasing limit of
	$$u_l:=\frac{\rho-\rho_j+M}{2^l}+\sum_{k=1}^{l}\frac{v_k-N_k}{2^k},$$
	is $(\beta+dd^c\rho_j)$-psh, such that  $u\neq-\infty$, and $u=-\infty$ on $K$.
\end{proof}

\begin{lem}\label{thm: extrem no mass}
If $K$ is  a compact subset of $X$ and  not  $(\beta+dd^c\rho_j)$-pluripolar, then $(\beta+dd^c\rho_j+dd^cV_{K,j}^*)^n$ puts no mass outside $K$.
\end{lem}

\begin{proof}
	The proof follows \cite[Theorem 5.2]{GZ05}.
Fix an arbitrary small ball $B\subseteq X-K$. Since  $K$  is not $(\beta+dd^c\rho_j)$-pluripolar, Lemma \ref{lem: non pluripolar char} implies that  $\sup_XV_{K_j}\textless+\infty$. 
By  Choquet's lemma, there is a   sequence $v_k\in \mbox{PSH}(X,\beta+dd^c\rho_j)$, $v_k\leq 0 $ on $K$, such that $(\sup v_k)^*=V_{K,j}^*$. 
We may assume the sequence  $\{v_k\}$ is increasing with respect to $k$ and $v_k\geq \rho-\rho_j$ (in particular, $v_k\in L^\infty(X)$). 
Applying \cite[Proposition 9.1]{BT82} to $v_k$ on $B$,  we can get $\tilde{v}_k\in \mbox{PSH}(X,\beta+dd^c\rho_j)$ such that 
	\begin{center}
		$\tilde{v}_k=v_k$ on $X\setminus B$,  $\tilde{v}_k\geq v_k$ on X, and $(\beta+dd^c\rho_j+dd^c\tilde{v}_k)^n=0$ on B.
	\end{center} 
Further,  $v_{k+1}\geq v_k$ on $X$ implies $\tilde{v}_{k+1}\geq\tilde{v}_k$ on X (see \cite[Corollary 12.5]{Dem95}). 
Since  $\tilde{v}_k=v_k\leq0$ on $K$, by the  definition of $V_{K,j}$, we have $\tilde{v}_k\leq V_{K,j}$. Hence   $\tilde{v}_k \nearrow V_{K,j}^*$  almost everywhere on $X$ (in Lebesgue measure).  By \cite[Proposition 5.2]{BT82},
\begin{align*} 
	(\beta+dd^c\rho_j+dd^c\tilde v_k)^n\rightarrow (\beta+dd^c\rho_j+dd^cV^*_{K,j})^n
	\end{align*}
weakly as measures on $B$, so  
$$(\beta+dd^c\rho_j+dd^cV^*_{K,j})^n=0\mbox{ on } B.$$ 
Since $B$ is arbitrary, the proof of Lemma \ref{thm: extrem no mass}	is complete.
\end{proof}

\begin{rem}\label{rem: extrem rhoj---rho}  In Lemma \ref{lem: non pluripolar char} and Lemma \ref{thm: extrem no mass}, we replace $\rho_j$ by $0$, and $V_{K,j}$ by $V_K:=\sup\{v\in\mbox{PSH}(X,\beta):v\leq0$ on $K \}$, the same conclusions also hold, i.e. 	$K\subset X$ is a $\beta$-pluripolar set
	if  and only if  $\sup_XV_{K}=+\infty$, and if $K$ is  a compact subset of $X$ and  not  $\beta$-pluripolar, then $(\beta+dd^cV_{K}^*)^n$ puts no mass outside $K$.
	\end{rem}
\subsection{Capacity estimates}
For any Borel set $E\subset X$, the capacity of $E$ with respect to $\beta_j$ is defined as
\[Cap_{\beta_j}(E):=\sup \left\{ \int_E (\beta_j+dd^c v)^n :v\in \mbox{PSH}(X,\beta_j), 0\leq v \leq 1\right \}.\]

\begin{prop}\label{prop:vol-cap est}
	There exist constants $\alpha,C\textgreater 0$, depending only on $(X,\omega,\beta,M)$, such that for any Borel set $E$, 
	$$vol_\omega(E) \leq C\exp(-\alpha Cap_{\beta_j}(E)^{-\frac{1}{n}}).$$
	
\end{prop}
\begin{proof}
	Because the volume  $vol_\omega(\cdot)$ and the capacity $Cap_{\beta_j}(\cdot)$ are inner regular, we may assume that  $E=K$ is a compact subset of X.
	We assume that $vol_\omega(K)> 0$, otherwise, it is trivial.
	Since $\rho-\rho_j \leq 0 $ on X, and $\beta+dd^c\rho_j+dd^c(\rho-\rho_j)\geq 0$, we see that 
	\[V_{K,j}\geq \rho-\rho_j  \geq-M,\]
	i.e., $V_{K,j}$ is uniformly bounded below.  
By Lemma \ref{lem: non pluripolar char}, $A_{K,j}:=\sup_X V_{K,j}<\infty$.
%
%
%
By Lemma \ref{thm: extrem no mass},   $(\beta+dd^c \rho_j+dd^c V_{K,j}^*)^n =0$ on  $X\setminus K$.  Moreover,   $V_{K,j}^*\leq 0$  almost everywhere on $K$ (in Lebesgue measure). Thus it follows that 
	\begin{align*}
		vol_\omega(K)&\leq \int_K e^{-\alpha V_{K,j}^*}\omega^n=e^{-\alpha A_{K,j}}\int_K e^{-\alpha (V_{K,j}^*-A_{K,j})}\omega^n\\
		&\leq Ce^{-\alpha A_{K,j}},
	\end{align*}
	where the second inequality follows from Proposition  \ref{prop:uniform est}.

Since  $A_{K,j}\geq \sup_X (\rho-\rho_j) \geq -M$, then $A_{K,j}+M\geq 0$. 

If  $A_{K,j}+M > 1$, we have 
\begin{align*}(A_{K,j}+M)^{-n}\int_X \beta^n&=(A_{K,j}+M)^{-n}\int_X (\beta+dd^c \rho_j+dd^c V_{K,j}^*)^n\\
		&=\int_K \left(\frac{\beta+dd^c \rho_j+dd^c V_{K,j}^*}{A_{K,j}+M}\right)^n\\
		&\leq\int_K \left(\frac{\beta_j+dd^c V_{K,j}^*}{A_{K,j}+M}\right)^n\\
		&	\leq \int_K \left(\beta_j+dd^c\left(\frac {V_{K,j}^*}{A_{K,j}+M}\right)\right)^n\\
		&= \int_K \left(\beta_j+dd^c\left(\frac {V_{K,j}^*-\inf_X V^{*}_{K,j}}{A_{K,j}+M}\right)\right)^n.
	\end{align*}
Then by the definition of  $Cap_{\beta_j}(\cdot)$, and the fact that 
	\[0\leq\frac {V_{K,j}^*-\inf_X V^{*}_{K,j}}{A_{K,j}+M}\leq 1, \]
	we get  $$(A_{K,j}+M)^{-n}\int_X \beta^n\leq Cap_{\beta_j}(K),$$
and thus  $$vol_\omega(K)\leq Ce^{\alpha M}\exp(-\alpha Cap_{\beta_j}(E)^{-\frac{1}{n}}).$$
	
If $A_{K,j}+M \leq 1$, we have 
	\begin{align*}
		\int_X \beta^n&=\int_X (\beta+dd^c \rho_j+dd^c V_{K,j}^*)^n\\
		&=\int_K (\beta+dd^c \rho_j+dd^c V_{K,j}^*)^n\\
		&\leq \int_K(\beta+dd^c \rho_j+\frac{1}{j}\omega+dd^c (V_{K,j}^*-\inf_XV_{K,j}^*))^n\\
		&\leq Cap_{\beta_j}(K).
	\end{align*}
	Thus $Cap_{\beta_j}(K)^{-\frac{1}{n}}\leq 1$. As a result,  
	\begin{align*}
		vol_\omega(K)\leq Ce^{\alpha M} \leq Ce^{\alpha (M+1)}\exp(-\alpha Cap_{\beta_j}(K)^{-\frac{1}{n}}).
	\end{align*}
	The proof of Proposition \ref{prop:vol-cap est} is complete.
\end{proof}

\begin{prop} \label{prop: cap est}Let $u\in \mbox{PSH}(X,\beta_j)$ and $\sup_X u=0.$ Then there exists a constant $C>0$, depending  only  on $(X,\omega,\beta,M)$, such that for any $t\geq  0$,
	\[Cap_{\beta_j}(\{u\textless -t\})\leq \frac{C}{t}.\]
\end{prop}
\begin{proof}
	For any $v\in \mbox{PSH}(X,\beta_j),$ $0\leq v \leq 1$,  
	\begin{align*}\int_{\{u \textless -t\}}(\beta_j+dd^cv)^n
	&	\leq \frac{1}{t} \int_X -u(\beta_j+dd^cv)^n \\
	&\leq \frac{1}{t} \int_X -(u-1)(\beta_j+dd^c(v-2))^n.
	\end{align*}
An application of  Proposition \ref{prop:CLN ineq} and Proposition \ref{prop:uniform est},  completes the proof of Proposition \ref{prop: cap est}.
%
\end{proof}

\begin{rem}\label{rem: est  cap uj--u}
For any Borel subset $E\subset X$, the capacity of $E$ with respect to $\beta$ is defined as 
	\[Cap_{\beta}(E):=\sup\left\{\int_E(\beta+dd^cv)^n:v\in \mbox{PSH}(X,\beta),\rho \leq v \leq \rho+1 \right\}.\]
The capacity of $E$ with respect to $\beta+dd^c\rho$ is defined as 
\[Cap_{\beta+dd^c \rho}(E):=\sup\{\int_E(\beta+dd^c\rho+dd^cv)^n: v\in \mbox{PSH}(X,\beta+dd^c\rho):0\leq v \leq1\}.\]
It is easy to see that they are coincide, i.e. 
\[Cap_{\beta}(E)=Cap_{\beta+dd^c \rho}(E).\]
	\end{rem}
By Remark \ref{rem: extrem rhoj---rho}, and  almost the same proof of Proposition \ref{prop:vol-cap est},
we can get the following
\begin{prop}\label{lem: vol-cap beta}
	There exist constants $\alpha,C\textgreater 0$, depending only on $(X,\omega,\beta,M)$, such that for any Borel set $E\subset X$, 
	$$vol_\omega(E) \leq C\exp(-\alpha Cap_{\beta}(E)^{-\frac{1}{n}}).$$
\end{prop}
\begin{rem} Demailly-Pali \cite[Lemma 2.9]{DP10} proved the same estimate as Proposition \ref{lem: vol-cap beta} under the assumption that there exists a continuous $\rho\in \mbox{PSH}(X,\beta)$.
	\end{rem}
\subsection{Non-pluripolar products and the classes  $\mathcal E^1(X,\beta)$ and $\mathcal E(X,\beta)$ }
Let $(X,\omega)$ be a compact Hermitian manifold of complex dimension $n$. Let $\{\beta\}\in H^{1,1}(X,\mathbb R)$ be a real $(1,1)$-class with smooth representative. Assume that there is bounded function $\rho\in \mbox{PSH}(X,\beta)$.

\begin{defn}[{\cite{BEGZ10}}]\label{defn: nn product}
Let  $\varphi \in \mbox{PSH}(X,\beta)$, the non-pluripolar product $ \left\langle(\beta+dd^c \varphi)^n \right\rangle$ is defined as 
$$\left\langle(\beta+dd^c \varphi)^n \right\rangle:=\lim_{k\rightarrow \infty} \mathds{1}_{\{\varphi \textgreater \rho-k\}}(\beta+dd^c\varphi^{(k)})^n,$$
where $\varphi^{(k)}:=\max\{\varphi,\rho-k\}$. 
\end{defn}
\begin{rem}
It is easy to see that  $O_k:=\{\varphi>\rho-k\}$ is a plurifine open subset, and for any compact subset  $K$ of $X$, 
\begin{align*}
	\sup_k\int_{K\cap O_k}(\beta+dd^c\varphi^{(k)})^n&\leq \sup_k\int_X(\beta+dd^c\varphi^{(k)})^n\\
	&=\int_X\beta^n<+\infty.
	\end{align*}
Thus the Definition \ref{defn: nn product} of non-pluripolar product is well defined due to \cite[Definition 1.1]{BEGZ10}, has basic properties in \cite[Proposition 1.4]{BEGZ10}.
	\end{rem}
\begin{defn}\label{defn: full ma mass}
Let  $\varphi \in \mbox{PSH}(X,\beta)$. 	We say $\varphi$ has full Monge-Amp\` ere mass if
	$$\int_X \left\langle(\beta+dd^c\varphi)^n\right\rangle=\int_X \beta^n.$$
	\end{defn}
\begin{rem}\label{rem: nonpp}
 For bounded $\varphi\in\mbox{PSH}(X,\beta)$, one can see that $\left\langle(\beta+dd^c \varphi)^n \right\rangle= (\beta+dd^c\varphi)^n $, thus has full Monge-Amp\`ere mass.
	\end{rem}

\begin{defn}
	We define the class $\mathcal{E}  ^1(X,\beta)$ to be the set of all  $\varphi \in \mbox{PSH}(X,\beta)$ such that
	$$\sup_k \int_X-\varphi^{(k)}(\beta+dd^c \varphi^{(k)})^n< \infty.$$
	 We define the class $\mathcal E(X,\beta)$ to be the set of all $\varphi\in\mbox{PSH}(X,\beta)$ with    full Monge-Amp\` ere mass.
\end{defn}
\begin{rem}\label{rem: e1 subset e}
It is easy to see that and bounded $\beta$-psh function is in $\mathcal E^1(X,\beta)$. Moreover, in Lemma \ref{lem: e1 energy est phi}, it will be proved  that  if $\varphi\in \mathcal E^1(X,\beta)$, then we have the following weak convergence
$$(\beta+dd^c\varphi^{(k)})^n\rightarrow\left\langle(\beta+dd^c\varphi)^n\right\rangle.$$
Thus $\mathcal{E}^1(X,\beta)\subseteq \mathcal{E}(X,\beta)$.

\end{rem}

%
%
%
\section{Weak solutions to  the degenerate CMA equations: the case of $\lambda=0$}

Let  $f\geq0,f\in L^p(X,\omega^n)$ for some $p>1$, and  $\int_X \beta^n=\int_X f \omega^n=1$. In this section, we will solve  the following degenerate CMA equations:
\begin{align*}
	(\beta+dd^c\varphi)^n=f\omega^n, \varphi\in \mbox{PSH}(X,\beta)\cap L^\infty(X).
	\end{align*}
By   \cite[Theorem 0.1]{KN15}, for every $j>1$, we can solve the following CMA  equations:
\begin{align}
	(\beta+dd^c \rho_j+\frac{1}{j}\omega+dd^cu_j)^n=c_jf\omega^n,\label{equ: ma equa}
\end{align}
where $u_j\in  \mbox{PSH}(X,\beta_j)\cap C(X)$,  $c_j > 0$, and  $\sup_X(u_j+\rho_j)=0.$

\subsection{Comparison of the Monge-Amp\`ere energy of $\beta_j$-psh functions}
In this section, we derive a  result on  the comparison of the Monge-Amp\`ere energy of $\beta_j$-psh functions.  For similar results, we refer to\cite[Lemma 2.3]{GZ07},  \cite[Proposition 1.6]{Ngu16}.
\begin{prop}\label{prop: compar energ} 
	Let  $u,v\in \mbox{PSH}(X,\beta_j)\cap L^{\infty}(X),$ and $u\leq v\leq-1$. Then there is a uniform constant   $C>0$, such that 
	\begin{align*}
		\int_X-v(\beta_j+dd^cv)^n\leq 2^k \int_X-u(\beta_j+dd^cu)^k\wedge (\beta_j+dd^cv)^{n-k}+\frac{C}{j}||u||_{\infty}^{2n}||v||_{\infty}^n.
	\end{align*} 
	In particular, 
	\begin{align*}
		\int_X-v(\beta_j+dd^cv)^n\leq 2^n \int_X-u(\beta_j+dd^cu)^n+\frac{C}{j}||u||_{\infty}^{2n}||v||_{\infty}^n.
	\end{align*}
\end{prop}
\begin{rem}If we  do not assume  $u,v\leq -1$, Proposition \ref{prop: compar energ} also holds by replacing $||u||_{\infty}$ and $||v||_{\infty}$ by $\max\{1,||u||_{\infty}\}$ and  $\max\{1,||v||_{\infty}\}$ respectively.
\end{rem}

\begin{proof} 
	The Proposition  is trivial  for $k=0$, so we proceed by induction on $k$.
	Assume the result holds for values up to $k>0$, we will prove that 
	\begin{align*}
		\int_X-v(\beta_j+dd^cv)^n\leq 2^{k+1} \int_X-u(\beta_j+dd^cu)^{k+1}\wedge (\beta_j+dd^cv)^{n-k-1}+\frac{C}{j}||u||_{\infty}^{2n}||v||_{\infty}^n.
	\end{align*}
 In the above and in the following equations, we use the convention that if $t<0$ is a negative integer, then the integral of $(\beta_j+dd^cu)^s\wedge(\beta_j+dd^cv)^t$ on $X$ equals $0$.

	It suffices to prove
	\begin{align*}
		\int_X-u(\beta_j+dd^cu)^k\wedge (\beta_j+dd^cv)^{n-k}\leq 2\int_X-u(\beta_j+dd^cu)^{k+1}\wedge (\beta_j+dd^cv)^{n-k-1}+\frac{C}{j}||u||_{\infty}^{2n}||v||_{\infty}^n.
	\end{align*}
	Write $\beta_j+dd^cv=\beta_j+dd^cu+dd^c(v-u)$. Then 
	\begin{align}
		\int_X-u(\beta_j+dd^cu)^k\wedge (\beta_j+dd^cv)^{n-k}&=\int_X-u(\beta_j+dd^cu)^{k+1}\wedge (\beta_j+dd^cv)^{n-k-1}\label{equ: ener ine 1}\\
		&+\int_Xudd^cu \wedge(\beta_j+dd^cu)^k\wedge (\beta_j+dd^cv)^{n-k-1}\notag\\
		&+\int_X-udd^cv \wedge (\beta_j+dd^cu)^k\wedge (\beta_j+dd^cv)^{n-k-1}\notag\\
		&=I+II+III.\notag
	\end{align}
	We need to estimate $II$ and $III$.
	
	\noindent \textbf{Claim 1. }$II\leq\frac{C}{j}||u||_{\infty}^{2n}||v||_{\infty}^n$.\\
		Note that $d^c(\beta_j+dd^cu)=d^c(\beta_j+dd^cv)=\frac{1}{j}d^c\omega$ and 	  $d (\beta_j+dd^cu)=d (\beta_j+dd^cv)=\frac{1}{j}d \omega$. From integration by parts, we have 
	\begin{align*}
		II&=-\int_Xdu \wedge d^cu \wedge(\beta_j+dd^cu)^k\wedge (\beta_j+dd^cv)^{n-k-1}\\
		&-\frac{k}{j}\int_X udu \wedge d^c\omega \wedge (\beta_j+dd^cu)^{k-1} \wedge (\beta_j+dd^cv)^{n-k-1}\\
		&-\frac{n-k-1}{j}\int_X udu \wedge d^c\omega \wedge (\beta_j+dd^cv)^{k} \wedge (\beta_j+dd^cv)^{n-k-2}.
	\end{align*}
	Since $du\wedge d^c u \geq0$,  the first term $\leq 0.$
	
	\begin{align*}
		II&\leq-\frac{k}{j}\int_X udu \wedge d^c\omega \wedge (\beta_j+dd^cu)^{k-1} \wedge (\beta_j+dd^cv)^{n-k-1}\\
		&-\frac{n-k-1}{j}\int_X udu \wedge d^c\omega \wedge (\beta_j+dd^cu)^{k} \wedge (\beta_j+dd^cv)^{n-k-2}\\
		&=-\frac{k}{2j}\int_X du^2 \wedge d^c\omega \wedge (\beta_j+dd^cu)^{k-1} \wedge (\beta_j+dd^cv)^{n-k-1}\\
		&-\frac{n-k-1}{2j}\int_X du^2 \wedge d^c\omega \wedge (\beta_j+dd^cu)^{k} \wedge (\beta_j+dd^cv)^{n-k-2}.
	\end{align*}
	The estimate of the term
	\[\int_X du^2 \wedge d^c\omega \wedge (\beta_j+dd^cu)^{k-1} \wedge (\beta_j+dd^cv)^{n-k-1}\]
	and the term 
	\[\int_X du^2 \wedge d^c\omega \wedge (\beta_j+dd^cu)^{k} \wedge (\beta_j+dd^cv)^{n-k-2}\] 
	are similar, thus in the following, we  give the detailed estimate of 
	\[\int_X du^2 \wedge d^c\omega \wedge (\beta_j+dd^cu)^{k-1} \wedge (\beta_j+dd^cv)^{n-k-1}.\]
	Through  integration by parts, we have 
	\begin{align*}&\int_X du^2 \wedge d^c\omega \wedge (\beta_j+dd^cu)^{k-1} \wedge (\beta_j+dd^cv)^{n-k-1}\\
		&=\int_X -u^2 dd^c\omega \wedge (\beta_j+dd^cu)^{k-1} \wedge (\beta_j+dd^cv)^{n-k-1}\\
		&-\frac{k-1}{j}\int_Xu^2 d\omega \wedge d^c\omega \wedge (\beta_j+dd^cu)^{k-2} \wedge (\beta_j+dd^cv)^{n-k-1}\\
		&-\frac{n-k-1}{j}\int_Xu^2 d\omega \wedge d^c\omega \wedge (\beta_j+dd^cu)^{k-1} \wedge (\beta_j+dd^cv)^{n-k-2}.
	\end{align*}
	From (\ref{equ:B cond}) and Proposition  \ref{prop:CLN ineq}, we have that the first term is controlled by $$C(X,\omega,\beta,M)||u||_{\infty}^{k+1}||v||_{\infty}^{n-k-1}.$$
	Similarly, the second term is controlled by $$\frac{C(X,\omega,\beta,M)}{j}||u||_{\infty}^{k}||v||_{\infty}^{n-k-1},$$ and the third term is controlled by $$\frac{C(X,\omega,\beta,M)}{j}||u||_{\infty}^{k+1}||v||_{\infty}^{n-k-2}.$$
	
	Continuing in this manner, we obtain that  $$\int_X du^2 \wedge d^c\omega \wedge (\beta_j+dd^cu)^{k} \wedge (\beta_j+dd^cv)^{n-k-2}$$
	is dominated by $$\frac{C(X,\omega,\beta,M)}{j}||u||_{\infty}^{2n}||v||_{\infty}^n,$$
	which completes the proof of \textbf{ Claim 1}.
	
	Now  it suffices to prove the following
	
	\noindent \textbf{ Claim 2.} $III\leq \int_X-u(\beta_j+dd^cu)^{k+1}\wedge (\beta_j+dd^cv)^{n-k-1}+\frac{C}{j}||u||_{\infty}^{2n}||v||_{\infty}^n.$\\
	An integration by parts yields:
	\begin{align*}&\int_X-udd^cv \wedge (\beta_j+dd^cu)^k\wedge (\beta_j+dd^cv)^{n-k-1}\\
		&=\int_X-vdd^cu \wedge (\beta_j+dd^cu)^k\wedge (\beta_j+dd^cv)^{n-k-1}\\
		&-2k\int_Xvdu \wedge d^c(\beta_j+dd^cu)\wedge (\beta_j+dd^cu)^{k-1}\wedge (\beta_j+dd^cv)^{n-k-1}\\
		&-2(n-k-1)\int_Xvdu \wedge d^c(\beta_j+dd^cv)\wedge (\beta_j+dd^cu)^{k}\wedge (\beta_j+dd^cv)^{n-k-2}\\
		&-k\int_X uvdd^c(\beta_j+dd^cu)\wedge (\beta_j+dd^cu)^{k-1}\wedge (\beta_j+dd^cv)^{n-k-1}\\
		&-(n-k-1)\int_X uvdd^c(\beta_j+dd^cv)\wedge (\beta_j+dd^cu)^{k}\wedge (\beta_j+dd^cv)^{n-k-2}\\
		&-k(k-1)\int_Xuv d(\beta_j+dd^cu)\wedge d^c(\beta_j+dd^cu)\wedge (\beta_j+dd^cu)^{k-2}\wedge(\beta_j+dd^cv)^{n-k-1}\\
		&-(n-k-1)(n-k-2)\int_Xuv d(\beta_j+dd^cv)\wedge d^c(\beta_j+dd^cv)\wedge (\beta_j+dd^cu)^{k}\wedge(\beta_j+dd^cv)^{n-k-3}\\
		&-2k(n-k-1)\int_Xuvd(\beta_j+dd^cu)\wedge d^c(\beta_j+dd^cv)\wedge(\beta_j+dd^cu)^{k-1}\wedge(\beta_j+dd^cv)^{n-k-2}.
	\end{align*}

	Firstly, we  compute that 
	\begin{align*}
		&\int_X-vdd^cu \wedge (\beta_j+dd^cu)^k\wedge (\beta_j+dd^cv)^{n-k-1}\\
		&=\int_X-v(\beta_j+dd^cu-\beta_j) \wedge (\beta_j+dd^cu)^k\wedge (\beta_j+dd^cv)^{n-k-1}\\
		&\leq \int_X-v(\beta_j+dd^cu) \wedge (\beta_j+dd^cu)^k\wedge (\beta_j+dd^cv)^{n-k-1}\\
		&\leq\int_X-u(\beta_j+dd^cu)^{k+1}\wedge (\beta_j+dd^cv)^{n-k-1}.
	\end{align*}
	In the following, we only need to prove that the other terms are dominated by $\frac{C}{j}||u||_{\infty}^{2n}||v||_{\infty}^n$.\\

	The other three terms are of the following types:
	\begin{itemize}
		\item [(T1):]$\int_X -uvd(\frac{1}{j}\omega)\wedge d^c(\frac{1}{j}\omega) \wedge (\beta_j+dd^cu)^{k-2}\wedge (\beta_j+dd^cv)^{n-k-1}$,
		with other similar terms is replacing  $\{k-2,n-k-1\}$   by $\{k-1,n-k-2\}$ and $\{k,n-k-3\}.$
		\item[(T2):] 	$\int_X -uvdd^c(\frac{1}{j}\omega) \wedge (\beta_j+dd^cu)^{k-1}\wedge (\beta_j+dd^cv)^{n-k-1}$, with other similar terms replacing $\{k-1,n-k-1\}$   by $\{k, n-k-2\}$.
		\item[(T3):] 	$\int_X-vdu\wedge d^c(\beta_j+dd^cv)\wedge (\beta_j+dd^cu)^k \wedge (\beta_j+dd^cv)^{n-k-2}$ and 	$\int_X-vdu\wedge d^c(\beta_j+dd^cu)\wedge (\beta_j+dd^cu)^{k-1} \wedge (\beta_j+dd^cv)^{n-k-1}.$
	\end{itemize}
	For (T1)  and (T2) terms, by Proposition \ref{prop:CLN ineq}, they are  controled by $\frac{C}{j }||u||_{\infty}^{2n}||v||_{\infty}^{n}.$\\
	%
	We next deal with (T3) terms.	
	Applying Cauchy-Schwarz type inequality (see \cite[Lemma 1.7]{Ngu16}),  
	\begin{align*}
		&	\int_X-vdu\wedge d^c\omega\wedge (\beta_j+dd^cu)^k \wedge (\beta_j+dd^cv)^{n-k-2}\\
		&\leq C(X,\omega,\beta,M)\int_X-vdu\wedge d^c u\wedge \omega \wedge (\beta_j+dd^cu)^k \wedge (\beta_j+dd^cv)^{n-k-2}\\
		&+C(X,\omega,\beta,M)\int_X-v \omega^2 \wedge (\beta_j+dd^cu)^k \wedge (\beta_j+dd^cv)^{n-k-2}\\
		&\leq C(X,\omega,\beta,M)||v||_{\infty}\int_Xdu\wedge d^c u \wedge (\beta_j+dd^cu)^k \wedge (\beta_j+dd^cv)^{n-k-2}\\
		&+C(X,\omega,\beta,M)||v||_{\infty}\int_X\omega^2 \wedge (\beta_j+dd^cu)^k \wedge (\beta_j+dd^cv)^{n-k-2}.
	\end{align*}
	Plugging $du\wedge d^cu= \frac{1}{2}dd^cu^2-udd^cu$ into the above inequality, and by Proposition \ref{prop:CLN ineq}, we can conclude that (T3) terms are also dominated by $\frac{C}{j }||u||_{\infty}^{2n}||v||_{\infty}^{n}.$ Then  \textbf{Claim 2} follows, and the proof of Proposition \ref{prop: compar energ} is complete.
	
	%
	%
	%
	
\end{proof}

\subsection{$L^\infty$-estimate  of $u_j$} \label{sect: Linfy est of uj}

\begin{lem}\label{lem: upper bound c_j}
	$c_j$ is uniformly bounded from above.
\end{lem}
\begin{proof}
Take a constant $A:=A(X,\omega,\beta)>0$ such that $\beta+\frac{1}{j}\omega\leq A\omega$ for all $j$. Then
$$(A\omega+dd^c(\rho_j+u_j))^n\geq (\beta_j+dd^cu_j)^n= c_jf\omega^n.$$ 
By  the mixed type inequality in the Hermitian setting (see \cite{Din09}, also \cite[Lemma 1.9]{Ngu16}),  
\[\int_X(A\omega+dd^c(\rho_j+u_j))\wedge \omega^{n-1}\geq \int_X(c_jf)^{1/n}\omega^n.\]
An integration by parts yields that 
\[\left|\int_Xdd^c(\rho_j+u_j)\wedge \omega^{n-1}\right|\leq C(X,\omega)\int_{X}-(\rho_j+u_j)\omega^n.\]
The last term is uniformly bounded above, by  \cite[Proposition 1.1]{Ngu16}.
Hence 
 \[\int_C(c_jf)^{1/n}\omega^n\leq C(X,\omega,\beta).\]
Since $\int_Xf\omega^n>0$, then $\int_Xf^{1/n}\omega^n>0$, and thus $c_j$ is uniformly bounded from above. 
\end{proof}

Let h: $\mathbb{R}^+\rightarrow(0,+\infty)$ be an increasing function which is called an admissible function, such that
$$\int_1^{+\infty}\frac{1}{x[h(x)]^{\frac{1}{n}}}\textless +\infty.$$
Define $F_h:\mathbb{R}^+\rightarrow (0,+\infty)$
	$$F_h(x):=\frac{x}{h(x^{-\frac{1}{n}})}.$$

\begin{lem}\label{lem: ui admissibel}
	Let $u_j$ be the solutions in (\ref{equ: ma equa}). Then there exists an admissible function $h$, such that  for any Borel subset $E\subset X$, 
\begin{align}\int_E(\beta_j+dd^cu)^n\leq F_h(Cap_{\beta_j}(E)).\label{equ: uj admis}
	\end{align}
	\end{lem}

%
\begin{proof}	
By the  H\"older's inequality, for any Borel subset $E\subset X$,
$$\int_E(\beta_j+dd^cu_j)^n=\int_E c_jf\omega^n \leq C_1||f||_pvol_{\omega}(E)^{\frac{1}{q}},$$
where $\frac{1}{p}+\frac{1}{q}=1$, $C_1$ is a uniform upper bound of \{$c_j\}$ in Lemma \ref{lem: ui admissibel}.
By the volume-capacity inequality (Proposition \ref{prop:vol-cap est}),	
	$$vol_{\omega}(E)^{\frac{1}{q}}\leq C_2\exp\left(-\frac{\alpha}{q} Cap_{\beta_j}(E)^{-\frac{1}{n}}\right),$$
	where $\alpha$ and $C_2$ depending only on $(X,\omega,\beta,M)$.
Hence 
$$\int_E(\beta_j+dd^cu_j)^n\leq C_1C_2||f||_p\exp\left(-\frac{\alpha}{q} Cap_{\beta_j}(E)^{-\frac{1}{n}}\right).$$

	Take $C_3\textgreater0$, $\frac{\alpha}{q} \textgreater \alpha_0\textgreater0$,  such that 
	$$x^n\leq C_3\exp\left(-(\alpha_0-\frac{\alpha}{q})x\right),\ x\textgreater0.$$
Then	$(\beta_j+dd^cu_j)^n$  satisifies (\ref{equ: uj admis}) for  the admissible function  
	$$h(x)=\frac{1}{C_1C_2C_3||f||_p}\exp(\alpha_0x).$$
The proof of Lemma \ref{lem: ui admissibel} is complete.

	\end{proof}
We need the  following theorem.
	\begin{thm}[{\cite[Theorem 5.3]{KN15}}]\label{thm: apriori kn}
Fix $0\textless \varepsilon \textless1$. Let $u,v\in \mbox{PSH}(X,\beta_j)\cap L^{\infty}(X)$ be such that $u\leq0$ and $-1\leq v \leq0$. Set $m(\varepsilon):=\inf_X[u-(1-\varepsilon) v]$, and $\varepsilon_0:=\frac{1}{3}\min\{\varepsilon^n,\frac{\varepsilon^3}{16B_j},4(1-\varepsilon)\varepsilon^n,4(1-\varepsilon)\frac{\varepsilon^3}{16B_j}   \}$. Suppose that $(\beta_j+dd^cu)^n$ satisfies  (\ref{equ: uj admis}) for an admissible function $h$. Then for $0\textless D \textless \varepsilon_0$,
		$$D\leq \kappa[Cap_{\beta_j}(U(\varepsilon,D))].$$
		where $U(\varepsilon,D):=\{u\textless (1-\varepsilon)v+m(\varepsilon)+D \}$, and the function $\kappa$ is defined on the integral $(0,cap_{\beta_j}(X))$ by the fomula
		$$\kappa(s^{-n}):=4C_n\left\{\frac{1}{[h(s)]^{\frac{1}{n}}}+\int_s^{+\infty}\frac{dx}{x[h(x)]^{\frac{1}{n}}}\right\},$$
		with a dimensional constant $C_n$.
		\begin{rem}\label{rem: adm apriori}
		In our situation, $\{cap_{\beta_j}(X)\}$  has a uniform lower bound, since $Cap_{\beta_j}(X)\geq \int_X(\beta+dd^c\rho_j+\frac{1}{j}\omega)^n\rightarrow \int_X(\beta+dd^c\rho)^n=1$.  Thus the open interval $(0,cap_{\beta_j}(X))$ contains a small interval $(0,\delta)$ for all $j$, and some $0<\delta<1$.
		\end{rem}
	\end{thm}

\begin{prop}\label{prop: linfty est uj}
	There exists $C=C(X,\omega,\beta,M,||f||_p)$ such that
	$$||u_j||_{L^\infty(X)}\leq C(\log j)^n.$$
\end{prop}
\begin{proof}
To apply Theorem \ref{thm: apriori kn}, we set $u=u_j$, $v=0$, $\varepsilon=\frac{1}{2}$, $D=\frac{\varepsilon_0}{2}=\frac{1}{768B_j}$, $h(s)=\frac{1}{C_1C_2C_3\|f\|_p}\exp(\alpha_0s)$ (see the proof of Lemma \ref{lem: ui admissibel}).  It is computed in \cite[Page 270]{Ngu16},  that there are uniform constants $C,\tilde a>0$ such that 
	$$\kappa(x)\leq C \exp(-\tilde{a}x^{-\frac{1}{n}}),$$ 
and the inverse function $\tilde h$ of $\kappa$ satisfies
	$$\tilde{h}(x)\geq\left (\frac{1}{\tilde{a}}\log\frac{C}{x}\right)^{-n}.$$
Now $U(\varepsilon,D)=U(\frac{1}{2},\frac{1}{768B_j})=\{u_j\textless \inf_Xu_j+\frac{1}{768B_j}\}$, by Lemma \ref{lem: ui admissibel}, we can apply   Theorem \ref{thm: apriori kn} to get 
	$$\frac{1}{768B_j}\leq\kappa\left[Cap_{\beta_j}\left(U\left(\frac{1}{2},\frac{1}{768B_j}\right)\right)\right],$$
hence
	$$\tilde{h}\left(\frac{1}{768B_j}\right)\leq Cap_{\beta_j}\left(U\left(\frac{1}{2},\frac{1}{768B_j}\right)\right).$$
Then 
\begin{align}\left(\frac{1}{\tilde{a}}\log(768CB_j)\right)^{-n}\leq Cap_{\beta_j}\left(U\left(\frac{1}{2},\frac{1}{768B_j}\right)\right).\label{equ: est of inf uj}
	\end{align}
	Since $\sup_X(\rho_j+u_j)=0$, $-M\leq \sup_Xu_j\leq M$. 
	We may assume $\sup_Xu_j-\inf_Xu_j\textgreater1$, otherwise there is nothing to prove.  
	Then  $\sup_Xu_j-\inf_Xu_j-\frac{1}{768B_j}>0$.
	By Proposition \ref{prop: cap est}, there exists a uniform constant $C'=C'(X,\omega,\beta,M)>0$, such that 
\begin{align*}
	Cap_{\beta_j}\left(U\left(\frac{1}{2},\frac{1}{768B_j}\right)\right)&=Cap_{\beta_j}\left(\left\{u_j\textless \inf_Xu_j+\frac{1}{768B_j}\right\}\right)\\
&=Cap_{\beta_j}\left(\left\{u_j-\sup_Xu_j \textless -\sup_Xu_j+\inf_Xu_j+\frac{1}{768B_j}\right\}\right)\\
&\leq C'\frac{1}{\sup_Xu_j-\inf_X u_j-\frac{1}{768B_j}}.
	\end{align*}
	Combining with (\ref{equ: est of inf uj}), we get that
	$$\left(\frac{1}{\tilde{\alpha}}\log(768CB_j)\right)^{-n}\leq C'\frac{1}{\sup_Xu_j-\inf_Xu_j-\frac{1}{768B_j}}.$$
Hence
\begin{align*}-\inf_Xu_j &\leq -\sup_Xu_j+\frac{1}{768B_j}+C'\left(\frac{1}{\tilde{a}}\log\left(768CB_j\right)\right)^n\\
&\leq M+1+C'\left(\frac{1}{\tilde{a}}\log(6144CBj^4)\right)^n.
	\end{align*}
The proof of Proposition \ref{prop: linfty est uj} is complete.
\end{proof}

\begin{cor}\label{cor: limit c_j}
	$c_j \rightarrow 1$ as $j \rightarrow \infty.$
\end{cor}
\begin{proof}Take a constant $A:=A(X,\omega,\beta)>0$ such that $\beta+\frac{1}{j}\omega\leq A\omega$ for all $j$. Then 
	\begin{align*}
		\int_X(\beta_j+dd^cu_j)^n&=\int_X(\beta+dd^c \rho_j+dd^cu_j)^n+\sum_{r=1}^{n} \tbinom{n}{r}\int_X(\beta+dd^c \rho_j+dd^cu_j)^{n-r}\wedge\left (\frac{1}{j}\omega\right)^r\\
		&=\int_X \beta^n+\sum_{r=1}^{n} \tbinom{n}{r}\int_X(\beta+dd^c \rho_j+dd^cu_j)^{n-r}\wedge \left(\frac{1}{j}\omega\right)^r\\
		&\leq \int_X \beta^n+\sum_{r=1}^{n} \tbinom{n}{r}\int_X(A\omega+dd^c \rho_j+dd^cu_j)^{n-r}\wedge \left(\frac{1}{j}\omega\right)^r
	\end{align*}
%
Since  $(\beta_j+dd^cu_j)^n=c_jf\omega^n$,  to estimate $c_j$, we only need to estimate the following integrals
	\begin{align*}
		\int_X(A\omega+dd^c \rho_j+dd^cu_j)^{n-r}\wedge \left(\frac{1}{j}\omega\right)^r, 1\leq  r\leq n.
	\end{align*}
By CLN inequality (see Proposition \ref{prop:CLN ineq}) and Proposition  \ref{prop: linfty est uj},  
\begin{align*}
	\int_X(A\omega+dd^c \rho_j+dd^cu_j)^{n-r}\wedge \left(\frac{1}{j}\omega\right)^r&\leq C(X,\omega,\beta,M)(\|\rho_j+u_j\|_{L^\infty(X)}+1)^{n-r}\left(\frac{1}{j}\right)^r\\
	&\leq C(X,\omega,\beta,M,\|f\|_p)\frac{(\log j)^{n^2}}{j}\rightarrow 0.
	\end{align*}
%
%
Hence 
	\begin{align*}
	\int_Xc_jf\omega^n=	\int_X(\beta_j+dd^cu_j)^n\xrightarrow{j\rightarrow \infty} \int_X \beta^n=\int_X f\omega^n.
	\end{align*}
As a result,  we get $c_j\rightarrow 1$ as $j\rightarrow\infty$.
\end{proof}

\subsection{$\mathcal{E}^1$-energy estimates} 
Since $\varphi_j:=\rho_j+u_j\in \mbox{PSH}(X,A\omega)$ for some large constant $A>0$,  and $\sup_X \varphi_j=0$, by Hartogs lemma, upto chosing  a  subsequence, we may assume $\varphi_j\rightarrow \varphi$ in $L^1(X,\omega^n)$ as $j\rightarrow \infty$, and $\varphi_j\rightarrow \varphi $ almost everywhere on $X$ (in Lebesgue measure) for some $\varphi\in \mbox{PSH}(X,A\omega)$ such that   $\sup_X \varphi =0$.

In this section, we estimate the $\mathcal E^1$-energy of $\varphi$ and prove that $\varphi$ is the  solution to the non-pluripolar Monge-Amp\`ere equation
 $$\left\langle(\beta+dd^c \varphi)^n\right\rangle=f\omega^n.$$

Set $$\Phi_j:=\left(\sup_{k\geq j}\varphi_k\right)^*.$$
Since for   $k\geq j$, $\varphi_k\in \mbox{PSH}(X,\beta+\frac{1}{k}\omega)\subseteq \mbox{PSH}(X,\beta+\frac{1}{j}\omega)$,
we have that
\begin{itemize} 
	\item $\varphi_j\leq \Phi_j$, $\sup_X \Phi_j=0$ and $\Phi_j \in \mbox{PSH}(X,\beta+\frac{1}{j}\omega)$.
	\item  $||\Phi_j||_{L^{\infty}(X)}\leq ||\varphi_j||_{L^{\infty}(X)}\leq C(\log j)^n$.
	\item  $\Phi_j\searrow \varphi$ on $X$, and thus $\varphi\in \mbox{PSH}(X,\beta)$.
\end{itemize}
Without of loss generality, we  assume $\rho_j\geq \rho \geq 1$. Set 
\begin{itemize} 
	\item $v_j:=\Phi_j-\rho_j\leq -1$ , then $v_j\in \mbox{PSH}(X,\beta_j)$,
	\item $u_j=\varphi_j-\rho_j \leq -1$, then $u_j\in \mbox{PSH}(X,\beta_j)$ and $u_j\leq v_j$,
	\item  $v_j^{(k)}:=\max\{v_j,-k\} \leq -1$, then  $v_j^{(k)}\in \mbox{PSH}(X,\beta_j)$ and $u_j\leq v_j^{(k)}$,
	\item $\Phi_j^{(k)}:=\max\{\Phi_j,\rho_j-k\}=\rho_j+v_j^{(k)} \in \mbox{PSH}(X,\beta+\frac{1}{j}\omega)$.
\end{itemize}


Applying Proposition  \ref{prop: compar energ} to $u_j\leq v_j^{(k)} \leq -1$, and by Proposition \ref{prop: linfty est uj} , we know that :
\begin{align*}
\int_X -\Phi_j^{(k)}\left(\beta+\frac{1}{j}\omega +dd^c \varphi_j^{(k)}\right)^n&\leq \int_X -v_j^{(k)}\left(\beta+\frac{1}{j}\omega+dd^c \Phi_j^{(k)}\right)^n\\
&=\int_X -v_j^{(k)}\left(\beta_j+dd^c v_j^{(k)}\right)^n\\
&\leq 2^n \int_X -u_j(\beta_j+dd^c u_j)^n+C\frac{(logj)^{3n^2}}{j}
\end{align*}
Since $(\beta_j+dd^cu_j)^n=c_jf\omega^n\leq 2f\omega^n$, for $j$ large,  $\varphi_j \in \mbox{PSH}(X,A\omega)$  and $\sup_X \varphi_j=0$, by \cite[Corollary 1.3]{Ngu16},  $$\int_X -\varphi_j f\omega^n\leq C(X,\omega, \beta, \|f\|_p).$$ 
Noting that  $||\rho_j||_{\infty}$ is uniformly bounded, then 
$$\int_X -\Phi^{(k)}_j\left(\beta+\frac{1}{j}\omega +dd^c \Phi_j^{(k)}\right)^n\leq C(X,\omega,\beta,M,\|f\|_p).$$ 
Similarly, one can get that 
$$\int_X -\Phi_j\left(\beta+\frac{1}{j}\omega +dd^c \Phi_j\right)^n\leq C(X,\omega,\beta,M,\|f\|_p).$$
We are on the way to prove the following
\begin{lem} \label{lem: e1 energy est phi}
	$\varphi\in \mathcal E^1(X,\beta)$, and $$\lim_{k\rightarrow \infty} (\beta+dd^c\varphi^{(k)})^n = \left\langle(\beta+dd^c\varphi)^n\right\rangle.$$
	In particular, 
	$$\int_X\left\langle(\beta+dd^c\varphi)^n\right\rangle=\int_X\beta^n=1.$$
	\end{lem}
\begin{proof}
Since $\Phi_j^{(k)}\searrow \varphi^{(k)}$ as $j\rightarrow  \infty$,   by Bedford-Taylor convergence theorem \cite[Theorem 2.4]{BT82},  
$$\left(\beta+\frac{1}{j}\omega+dd^c \Phi_j^{(k)}\right)^n \rightarrow(\beta+dd^c\varphi^{(k)})^n$$ in the weak sense of currents   as $j\rightarrow \infty$.
By \cite[Lemma 2.3]{BT82}, 
\begin{align*}
\limsup_j \int_X \Phi_j^{(k)}\left(\beta+\frac{1}{j}\omega +dd^c \Phi_j^{(k)}\right)^n	\leq \int_X \varphi^{(k)}(\beta+dd^c \varphi^{(k)})^n.
\end{align*}
Thus 
\begin{align*}
\int_X -\varphi^{(k)}(\beta+dd^c \varphi^{(k)})^n&\leq -	\limsup_j \int_X \Phi_j^{(k)}\left(\beta+\frac{1}{j}\omega +dd^c \Phi_j^{(k)}\right)^n	\\
&\leq \sup_j \int_X -\Phi_j^{(k)}\left(\beta+\frac{1}{j}\omega +dd^c \Phi_j^{(k)}\right)^n\\
&\leq C(X,\omega,\beta,M,\|f\|_p).
\end{align*}
This shows that $\varphi\in \mathcal E^1(X,\beta)$.


For any test function $g\in \mathcal C^\infty(X)$, we are going to show that $$\lim_{k\rightarrow \infty}<(\beta+dd^c\varphi^{(k)})^n,g>=<\left\langle(\beta+dd^c\varphi)^n\right\rangle,g>.$$
	By definition, $$\int_Xg\left\langle(\beta+dd^c\varphi)^n\right\rangle=\lim_{k\rightarrow \infty}\int_{\{\varphi \textgreater \rho-k\}}g(\beta+dd^c\varphi^{(k)})^n.$$
	So it suffices to show that $$\lim_{k\rightarrow \infty}\int_{\{\varphi \leq \rho-k\}}g(\beta+dd^c\varphi^{(k)})^n=0.$$ 
	Since 
	$$E_1(\varphi):=\sup_k\int_X-(\varphi^{(k)}-\rho)(\beta+dd^c\varphi^{(k)})^n \textless +\infty,$$
	we have 
	$$\int_{\{\varphi \leq \rho-k\}}(\beta+dd^c\varphi^{(k)})^n=\frac{1}{k}\int_{\{\varphi \leq \rho-k \}}k(\beta+dd^c\varphi^{(k)})^n$$
	$$= \frac{1}{k}\int_{\{\varphi \leq \rho-k\}}-(\varphi^{(k)}-\rho)(\beta+dd^c\varphi^{(k)})^n\leq \frac{1}{k}E_1(\varphi).$$
	Therefore 
	$$\left|\int_{\{\varphi \leq \rho-k\}}g(\beta+dd^c\varphi^{(k)})^n \right|
	$$$$\leq \|g\|_{L^\infty(X)}\int_{\{\varphi \leq \rho-k\}}(\beta+dd^c\varphi^{(k)})^n\leq \frac{\|g\|_{L^\infty(X)}}{k}E_1(\varphi).$$
	This finishes the proof.
\end{proof}

\begin{thm}\label{lem: conv phi}
$\left(\beta+\frac{1}{j}\omega+dd^c\Phi_j\right)^n\rightarrow \left\langle(\beta+dd^c \varphi)^n\right\rangle=f\omega^n$. 
\end{thm}
\begin{proof}
First, we prove that 
\begin{align*}
	\lim_{k\rightarrow \infty}\left(\beta+\frac{1}{j}\omega+dd^c\Phi_j^{(k)}\right)^n=\left(\beta+\frac{1}{j}\omega+dd^c\Phi_j\right)^n.
	\end{align*}
In fact, for any  test function $g\in \mathcal C^{\infty}(X)$ and $k>M$,  

\begin{align*}
&	\left|\int_Xg\left(\beta+\frac{1}{j}\omega+dd^c\Phi_j^{(k)}\right)^n-g\left(\beta+\frac{1}{j}\omega+dd^c\Phi_j\right)^n\right|\\
&=\left|\int_{\{\Phi_j \leq \rho_j-k\}}g\left(\beta+\frac{1}{j}\omega+dd^c\Phi_j^{(k)}\right)^n-g\left(\beta+\frac{1}{j}\omega+dd^c\Phi_j\right)^n\right|\\
&\leq \frac{||g||_{L^\infty(X)}}{k-M}  \int_{\{\Phi_j \leq \rho_j-k\}}(-\Phi_j^{(k)})\left(\beta+\frac{1}{j}\omega+dd^c\Phi_j^{(k)}\right)^n+(-\Phi_j)\left(\beta+\frac{1}{j}\omega+dd^c\Phi_j\right)^n\\
&\leq \frac{||g||_{L^\infty(X)}}{k-M}  \int_X(-\Phi_j^{(k)})\left(\beta+\frac{1}{j}\omega+dd^c\Phi_j^{(k)}\right)^n+(-\Phi_j)\left(\beta+\frac{1}{j}\omega+dd^c\Phi_j\right)^n\\
&\leq \frac{C(X,\omega,\beta,M,\|f\|_p)||g||_{L^\infty(X)}}{k-M}.
\end{align*}

Since  $\left(\beta+\frac{1}{j}\omega+dd^c \Phi_j^{(k)}\right)^n\rightarrow  (\beta+dd^c \varphi^{(k)})^n$ as $j \rightarrow  \infty$,
$$\int_Xg\left(\beta+\frac{1}{j}\omega+dd^c \Phi_j^{(k)}\right)^n \rightarrow  \int_Xg(\beta+dd^c\varphi^{(k)})^n,$$
thus 
\begin{align*}
	\limsup_{j\rightarrow \infty} \left|\int_Xg(\beta+dd^c\varphi^{(k)})^n-g\left(\beta+\frac{1}{j}\omega+dd^c\Phi_j\right)^n\right| \leq \frac{C}{k-M}.
	\end{align*}
Combining with  Lemma \ref{lem: e1 energy est phi}, we have that 
\begin{align*}
	\limsup_{j\rightarrow \infty} \left|\int_Xg \left\langle(\beta+dd^c\varphi)^n \right\rangle  -g(\beta+\frac{1}{j}\omega+dd^c\Phi_j)^n\right| \leq 0.
	\end{align*}
This proves that 
\begin{align}\left(\beta+\frac{1}{j}\omega+dd^c\Phi_j\right)^n\rightarrow \left\langle(\beta+dd^c \varphi)^n\right\rangle.\label{equ: nonpluri solu}
	\end{align}

From  $\beta+\frac{1}{j}\omega+dd^c \varphi_k\geq \beta+\frac{1}{k}\omega +dd^c \varphi_k,$ for $k\geq j$, we know that 
\begin{align*}
	\left(\beta+\frac{1}{j}\omega+dd^c \varphi_k\right)^n\geq \left(\beta+\frac{1}{k}\omega+ dd^c \varphi_k\right)^n=c_kf\omega^n.
\end{align*}
By \cite[Proposition 2.8]{BT76},
\begin{align*}
	\left(\beta+\frac{1}{j}\omega+dd^c \max\{\varphi_j,...,\varphi_k\}\right)^n \geq \min\{c_j,...,c_k\}f\omega^n  \mbox{~for~} k\geq j.
\end{align*}
Since $\max\{\varphi_j,...,\varphi_k\}$ increases to $\varphi_j$ almost everywhere, the Bedford-Taylor monotone convergence theorem   \cite[Theorem 7.4]{BT82} implies that 
\begin{align*}
	\left(\beta+\frac{1}{j}\omega+dd^c \Phi_j\right)^n\geq \inf_{k\geq j}\{c_k\}f\omega^n.
\end{align*}
Then  from (\ref{equ: nonpluri solu}),
we obtain
 $$\left\langle(\beta+dd^c \varphi)^n \right\rangle \geq \lim_{j\rightarrow\infty}\inf_{k\geq j}c_kf\omega^n=f\omega^n.$$
However, from Lemma \ref{lem: e1 energy est phi}, $\int_X \left\langle(\beta+dd^c \varphi)^n \right\rangle =\int_X f \omega^n=1$,  thus  $\left\langle(\beta+dd^c \varphi)^n\right\rangle=f\omega^n$.
\end{proof}

\subsection{$L^\infty$-estimate of $\varphi$}
In this section, we prove that $\varphi$ is a bounded $\beta$-psh function, and then the non-pluripolar product $\langle(\beta+dd^c \varphi)^n\rangle$ coincides with the Bedford-Taylor product $(\beta+dd^c\varphi)^n$, thus we finish the proof of the existence part of the degenerate CMA equation
\begin{align*}
	(\beta+dd^c\varphi)^n=f\omega^n, \varphi\in \mbox{PSH}(X,\beta)\cap L^\infty(X).
\end{align*}

Since $\varphi\in \mathcal E^1(X,\beta)$ (see Lemma \ref{lem: e1 energy est phi}), and $\mathcal E^1(X,\beta)\subset \mathcal E(X,\beta)$ (see Remark \ref{rem: e1 subset e}), 
 it suffices to prove $L^{\infty}$-estimate in  $\mathcal{E}(X,\beta)$. 

The following comparison principle is an important tool in the sequel.
\begin{prop}\label{prop: np comparison principle}
Let $\varphi,\psi \in \mbox{PSH}(X,\beta)$, then we have following inequality:
	$$\int_X\beta^n-\int_X\left\langle(\beta+dd^c\psi)^n\right\rangle+\int_{\{\psi \textless \varphi\}}\left\langle(\beta+dd^c\psi)^n\right\rangle \geq\int_{\{\psi \textless \varphi\}}\left\langle(\beta+dd^c\varphi)^n\right\rangle.$$
	In particular, if $\varphi,\psi \in \mathcal{E}(X,\beta)$ , then we have
	$$\int_{\{\varphi\textless \psi\}}\left\langle(\beta+dd^c\psi)^n\right\rangle\leq \int_{\{\varphi \textless \psi\}}\left\langle(\beta+dd^c\varphi)^n\right\rangle.$$
\end{prop}

\begin{proof}Since
\begin{align*} 
	\int_X\beta^n& \geq \int_X\left\langle(\beta+dd^c\max\{\varphi,\psi\})^n\right\rangle\\
	&\geq \int_{\{\varphi \textless  \psi\}}\left\langle(\beta+dd^c\max\{\varphi,\psi\})^n\right\rangle+\int_{\{\psi \textless \varphi\}}\left\langle(\beta+dd^c\max\{\varphi,\psi\})^n\right\rangle\\
&=\int_{\{\varphi \textless  \psi\}}\left\langle(\beta+dd^c\psi)^n\right\rangle+\int_{\{\psi \textless \varphi\}}\left\langle(\beta+dd^c\varphi)^n\right\rangle\\
&=\int_X\left\langle(\beta+dd^c\psi)^n\right\rangle-\int_{\{\psi \leq \varphi\}}\left\langle(\beta+dd^c\psi)^n\right\rangle+\int_{\{\psi \textless \varphi\}} \left\langle(\beta+dd^c\varphi)^n\right\rangle,
\end{align*}
then
	$$\int_X\beta^n-\int_X\left\langle(\beta+dd^c\psi)^n\right\rangle+\int_{\{\psi \leq \varphi\}}\left\langle(\beta+dd^c\psi)^n\right\rangle \geq \int_{\{\psi \textless \varphi\}}\left\langle(\beta+dd^c\varphi)^n\right\rangle.$$
	Replace   $\psi$ by $\psi+\varepsilon$, $\varepsilon\textgreater0$.  Since $\{\psi+\varepsilon \textless \varphi\}$ increases to $\{\psi \textless \varphi\}$, $\{\psi+\varepsilon \leq \varphi\}$ increases to $\{\psi \textless \varphi\}\cup\{\psi=\varphi=-\infty\}$, and  non-pluripolar product   puts  no mass on pluripolar set $\{\varphi=\psi=-\infty\}$, when $\varepsilon\searrow 0$, we have 
	$$\int_X\beta^n-\int_X\left\langle(\beta+dd^c\psi)^n\right\rangle+\int_{\{\psi\textless\varphi\}}\left\langle(\beta+dd^c\psi)^n\right\rangle \geq \int_{\{\psi\textless\varphi\}}\left\langle(\beta+dd^c\varphi)^n\right\rangle.$$

\end{proof}
\begin{lem}\label{lem: bounded solution cap est level}
	Let $\varphi,\psi \in \mathcal{E}(X,\beta)$, 
	 $\psi\leq \rho$. Then for any $0\leq t \leq 1,$ $s\in \mathbb{R}$, 
	$$t^nCap_{\beta}(\{\varphi \textless \psi -t-s\})\leq \int_{\{\varphi \textless (1-t)\psi+t\rho-s\}}\langle (\beta+dd^c\varphi)^n\rangle.$$ 
\end{lem}
\begin{proof}
	For any  $v \in \mbox{PSH}(X,\beta+dd^c \rho)$, $0\leq v \leq 1$, we have $(1-t)\psi+t(\rho+v) -s-t\in \mathcal E(X,\beta)$ and 
	$$\{\varphi \textless \psi-s-t\} \subseteq \{\varphi \textless (1-t)\psi+t(\rho+v) -s-t \} \subseteq\{\varphi \textless (1-t)\psi+t\rho-s\}.$$ 
	By \cite[Proposition 1.4(c)]{BEGZ10}, 
	$$t^n\langle(\beta+dd^c\rho+dd^cv)^n\rangle \leq \langle(t(\beta+dd^c\rho+dd^cv)+(1-t)(\beta+dd^c\psi))^n\rangle.$$ 
Applying Proposition \ref{prop: np comparison principle}, we get 
\begin{align*}
	&	t^n\int_{\{\varphi \textless \psi -t-s\}}(\beta+dd^c\rho+dd^cv)^n\\
&=	t^n\int_{\{\varphi \textless \psi -t-s\}}\langle(\beta+dd^c\rho+dd^cv)^n\rangle \\
&\leq 	t^n\int_{ \{\varphi \textless (1-t)\psi+t(\rho+v) -s-t \} }\langle(\beta+dd^c\rho+dd^cv)^n\rangle \\
&\leq 	t^n\int_{ \{\varphi \textless (1-t)\psi+t(\rho+v) -s-t \} }	 \langle(t(\beta+dd^c\rho+dd^cv)+(1-t)(\beta+dd^c\psi))^n\rangle\\
&\leq 	\int_{ \{\varphi \textless (1-t)\psi+t(\rho+v) -s-t \} }	 \langle(\beta+dd^c \varphi)^n\rangle\\
&\leq 	\int_{\{\varphi \textless (1-t)\psi+t\rho-s\}}\langle(\beta+dd^c \varphi)^n\rangle.
	\end{align*}
The proof  of Lemma \ref{lem: bounded solution cap est level}  is complete.
\end{proof}

\begin{lem}\label{lem: bounded solution cap est level t}
	Let $\varphi \in  \mbox{PSH}(X,\beta)$ and $\sup_X\varphi=0$. Then there exists a uniform constant $C=C(X,\omega,\beta,M)>0$, such that for 
	any $t\textgreater0$
	$$Cap_{\beta}(\{\varphi \textless \rho-t\})\leq\frac{C}{t}.$$
\end{lem}
\begin{proof}
	Take $v\in  \mbox{PSH}(X,\beta+dd^c\rho)$, $0\leq v \leq1$, 
	$$\int_{\{\varphi \textless \rho-t\}}(\beta+dd^c\rho+dd^cv)^n\leq \frac{1}{t}\int_X-(\varphi-\rho)(\beta+dd^c\rho+dd^cv)^n.$$
	Since $\rho$ is bounded on X, we only need to give $\int_X-\varphi(\beta+dd^c\rho+dd^cv)^n$ a uniform bound. But note that $\beta\leq A\omega$ for some $A$ large enough. Therefore $\varphi,(\rho+v) \in  \mbox{PSH}(X,A\omega)$, So we have
	$$\int_X-\varphi(\beta+dd^c\rho+dd^cv)^n\leq \int_X-\varphi(A\omega+dd^c\rho+dd^cv)^n.$$
	By   Proposition \ref{prop:uniform est}   and Proposition \ref{prop:CLN ineq} (actually the proof), there is a uniform constant  $C=C(X,\omega,\beta,M)>0$, such that
	$$\int_X-\varphi(A\omega+dd^c\rho+dd^cv)^n\leq C.$$
	Thus the proof of Lemma \ref{lem: bounded solution cap est level} is complete.

\end{proof}

\begin{lem}[cf. {\cite[Lemma 2.4]{EGZ09}}]\label{lem: egz func}
	Let $f:\mathbb{R}^+ \rightarrow \mathbb{R}^+$ be a decreasing right-continuos function such that $\lim_{s\rightarrow +\infty}f(s)=0$. Assume that there exists $a, B$ such that $f$ satisfies:
	$$tf(s+t)\leq B[f(s)]^{1+a},\forall s\textgreater0,\forall  0\leq t\leq1.$$
	Then there exists $S=S(a,B)\in \mathbb{R}^+$, such that $f(x)=0$ for all $x\geq S$. 
\end{lem}

No we are on the way to   get the $L^\infty$ estimate of $\varphi$.
\begin{thm}\label{thm: main 2 in sect}
Let $0\leq f\in L^p(X,\omega^n)$, $\int_Xf\omega^n>0$. Let  $\varphi\in \mathcal{E}(X,\beta)$,  $\sup_X\varphi=0$, such that 
	$$\langle(\beta+dd^c\varphi)^n\rangle= f\omega^n.$$
Write $\varphi=\rho+u$, then    $||u||_{\infty}\leq S=S(X,\omega,\beta,||f||_p)>0$, thus  $\varphi\in L^\infty(X)$. In particular, 
	$$ (\beta+dd^c\varphi)^n= f\omega^n.$$
\end{thm}
\begin{proof}
We follow  the method in \cite{EGZ09}.
	Set $\mu=f\omega^n$. For any Borel set $E\subset X$, by H\"older's inequality and Lemma \ref{lem: vol-cap beta},  
	$$\mu(E)\leq||f||_pvol(E)^{\frac{1}{q}}\leq||f||_p(C\exp(-\alpha Cap_\beta(E)^{-\frac{1}{n}}))^{\frac{1}{q}}.$$
	Easy to see that $e^{-\frac{\alpha}{q}x^{-\frac{1}{n}}}\leq Ax^{1+a}$ for some constant $A,a$ depending only on $\alpha, q, n$. Hence for some  constant $C=C(X,\omega,\beta,||f||_p)>0$, we have 
	$$\mu(E)\leq CCap_{\beta}(E)^{1+a}.$$
	Now we set
	$$f(s)=[Cap_{\beta}(\{\varphi\textless\rho -s\})]^{\frac{1}{n}}=[Cap_{\beta}(\{u\textless -s\})]^{\frac{1}{n}}.$$
	By Lemma \ref{lem: bounded solution cap est level},  for all $s\textgreater0,$ $0\leq t\leq 1$, we have 
	$$tf(s+t)\leq \left[\int_{\{\varphi\textless \rho-s\}}(\beta+dd^c\rho+dd^cu)^n\right]^{\frac{1}{n}}\leq \left[\int_{\{u\textless -s\}}d\mu\right]^{\frac{1}{n}} $$
	$$\leq C[Cap_{\beta}(\{u\textless-s\})]^{\frac{1+a}{n}}=Cf(s)^{1+a}.$$
	Since $f$ is right-continous, and from Lemma \ref{lem: bounded solution cap est level t}, we know that $\lim_{s\rightarrow+\infty}f(s)=0$. Now, we use Lemma \ref{lem: egz func}  to conclude that there exists $S=S(X,\omega,\beta,||f||_p)>0$ such that 
	$Cap_{\beta}(\{u\textless-S\})=0$ for all $s\geq S_\infty$.
	This means that the sets $\{u<-s\}$ are empty if $s\geq S_\infty$, hence 
	$$||u||_{\infty}\leq S=S(X,\omega,\beta,||f||_p).$$
\end{proof}
To sum up, we get the following existence theorem for degenerate CMA equations. 
\begin{thm}\label{thm: main 1'}Let $(X,\omega)$ be a compact Hermitian manifold  of complex dimension $n$.  Let $\beta$ be a smooth real $(1,1)$-form on $X$ such that there exists  $\rho\in \mbox{PSH}(X,\beta)\cap L^\infty(X)$. Let $0\leq f\in L^p(X,\omega^n)$, $p>1$, be such that $\int_Xf\omega^n=\int_X\beta^n>0$. Then there exists a  bounded real-valued function $\varphi\in \mbox{PSH}(X,\beta )\cap L^\infty(X)$, satisfying
	\[(\beta+dd^c\varphi)^n=f\omega^n\]
	in the weak sense of currents, and $\|\varphi\|_{L^\infty(X)}\leq C(X,\omega,\beta, M,\|f\|_p)$.
\end{thm}
\subsection{Uniqueness of the solutions }
In this section, we prove the uniqueness of the  solutions of the   degenerate CMA equations. The proof follows the method introduced  by Dinew in \cite{Din09-2} and by  Boucksom-Eyssidieux-Guedj-Zeriahi in \cite{BEGZ10}.
\begin{thm}\label{thm: uniq of lambda=0}
Let Let $0\leq f\in L^p(X,\omega^n)$, $p>1$, be such that  $\int_Xf\omega^n=\int_X\beta^n=1.$ Then there exists at most one bounded solution to the following degenerate complex Monge-Amp\`ere equation
	$$(\beta+dd^c\varphi)^n=f\omega^n,	\sup_X\varphi=0.$$
\end{thm}

\begin{proof}
	Let $\varphi_1,\varphi_2$ be two such solutions.  Let $d\mu:=f\omega^n$. 
	Since $d\mu$ puts no mass on pluripolar subsets, the set $t\in \mathbb R$ such that 
	$$\mu(\{\varphi_1=\varphi_2+t\})>0$$
coincides with the discontinuity locus of the non-decreasing function 
$$t\mapsto \mu(\{\varphi_1<\varphi_2+t\})$$
 and is thus at most countable.

\noindent\textbf{Caim.}	For any $t\in\mathbb R$, either $\mu(\{\varphi_1<\varphi_2+t\})=0$ or $\mu(\{\varphi_1<\varphi_2+t\})=\mu(X)$. In particular, there exists $t_0\in \mathbb R$, such that $\mu(\{\varphi_1=\varphi_2+t_0\})=\mu(X)$, i.e. $\varphi_1=\varphi_2+t_0$ $\mu$-almost everywhere.

To prove the \textbf{Claim}, we argue by contradiction. Suppose to the contrary, $\mu(\{\varphi_1<\varphi_2+t\})\in (0,\mu(X))$ for some $t\in \mathbb R$. We consider the measure 
$$\hat \mu=(1+\varepsilon)^n\mathds 1_{\{\varphi_1<\varphi_2+t\}}\mu+\frac{1}{A}\mathds 1_{\{\varphi_1\geq \varphi_2+t\}}\mu$$
such that $\hat\mu(X)=\mu(X)$.	Let $F:=\left((1+\varepsilon)^n\mathds 1_{\{\varphi_1<\varphi_2+t\}}+\frac{1}{A}\mathds 1_{\{\varphi_1\geq \varphi_2+t\}}\right)f$. Then $0\leq F\in L^p(X,\omega^n)$,  $d\hat\mu=\left((1+\varepsilon)^n\mathds 1_{\{\varphi_1<\varphi_2+t\}}+\frac{1}{A}\mathds 1_{\{\varphi_1\geq \varphi_2+t\}}\right)d\mu=F\omega^n$, and $\int_XF\omega^n=\hat\mu(X)=\mu(X)=1$.
By Theorme \ref{thm: main 1'}, we can solve the equation $(\beta+dd^c\rho)^n=\hat\mu$, $\varphi\in \mbox{PSH}(X,\beta)\cap L^\infty(X)$, such that $\varphi\leq \varphi_1$. For $0<\delta<1$, we set 
\[O_\delta:=\{\varphi_1<(1-\delta)(\varphi_2+t)+\delta\varphi\}\mbox{~and~} O:=\{\varphi_1<\varphi_2+t\}.\]
Hence  $O_\delta\subset \{\varphi_1-\varphi_2-t<\frac{\delta}{1-\delta}(\varphi-\varphi_1)\}\subset O$, and $\mu(O_\delta)\nearrow \mu(O)>0$ as $\delta\searrow 0$. In particular, $\mu(O_\delta)>0$ for small $\delta>0$.
By the comparision principle (see Proposition \ref{prop: np comparison principle}),
$$\int_{O_\delta}(\beta+dd^c((1-\delta)(\varphi_2+t)+\delta\varphi))^n\leq \int_{O_\delta}(\beta+dd^c\varphi_1)^n.$$
Since $d\mu$ and $ d\hat\mu=\left((1+\varepsilon)^n\mathds 1_{\{\varphi_1<\varphi_2+t\}}+\frac{1}{A}\mathds 1_{\{\varphi_1\geq \varphi_2+t\}}\right)d\mu$ are positive measures on $X$ and vanish on all pluripolar sets,
by \cite[Theorem 1.3]{Din09}, on $O$,
$$(\beta+dd^c((1-\delta)(\varphi_2+t)+\delta\varphi))^n\geq (1+\delta\varepsilon)^nd\mu.$$
Then we obtain 
\begin{align*}
 \mu(O_\delta)&=\int_{O_\delta}(\beta+dd^c\varphi_1)^n\\
 &\geq \int_{O_\delta}(\beta+dd^c((1-\delta)(\varphi_2+t)+\delta\varphi))^n\\
 &\geq (1+\delta\varepsilon)^n\mu(O_\delta)>\mu(O_\delta).
\end{align*}
This is a contradiction, thus we complete the proof of the \textbf{Claim}. 

Next we prove that $\{\varphi_1=\varphi_2+t_0\}=X$. Otherwise, without loss of generality, we may assume $\{\varphi_1<\varphi_2+t_0\}\neq \emptyset$. Let $dV$ be any smooth volume form such that $vol(X)=\mu(X)$. By Theorem \ref{thm: main 1'}, we can solve the equation $(\beta+dd^c\varphi)=dV$ such that $\varphi\leq \varphi_1$. Set
\[O'_\delta:=\{\varphi_1<(1-\delta)(\varphi_2+t_0)+\delta\varphi\},\mbox{ and }O':=\{\varphi_1<\varphi_2+t_0\}.\]
Observe that $O'$ has positive measure (in Lebesgue measure) and $\mu(O'_\delta)=0$ for $0<\delta<1$.  Since  $O'_\delta\nearrow O'$ as $\delta\searrow 0$, then $O'_\delta$ has positive measure (in Lebesgue measure) for small $0<\delta<1$. Applying comparison principle (see Proposition \ref{prop: np comparison principle}), we have 
\begin{align*}
	0<\delta^n\int_{O'_\delta}dV&\leq \int_{O'_\delta}(\beta+dd^c(1-\delta)(\varphi_2+t_0)+\delta\varphi))^n\\
&\leq \int_{O'_\delta}(\beta+dd^c\varphi_1)^n=\int_{O'_\delta}d\mu=\mu(O'_\delta)=0.	
\end{align*}
This contradiction shows that $\{\varphi_1<\varphi_2+t_0\}=\emptyset$. Similarly, one can prove that $\{\varphi_1>\varphi_2+t_0\}=\emptyset$. As a result, we get $\varphi_1=\varphi_2+t_0$ on $X$. Since $\sup_X\varphi_1=\sup_X\varphi_2=0$, then $t_0=0$, i.e. $\varphi_1=\varphi_2$ on $X$.  Thus the proof  of Theorem \ref{thm: uniq of lambda=0} is complete.
\end{proof}
To conclude, we finish the proof of our first main theorem.
\begin{thm}[=Theorem \ref{thm: main 1}]\label{thm: main 1-2}Let $(X,\omega)$ be a compact Hermitian manifold  of complex dimension $n$.  Let $\beta$ be a smooth real $(1,1)$-form on $X$ such that there exists  $\rho\in \mbox{PSH}(X,\beta)\cap L^\infty(X)$. Let $0\leq f\in L^p(X,\omega^n)$, $p>1$, be such that $\int_Xf\omega^n=\int_X\beta^n>0$. Then there exists a unique  real-valued function $\varphi\in \mbox{PSH}(X,\beta )\cap L^\infty(X)$, satisfying
	\[(\beta+dd^c\varphi)^n=f\omega^n\]
	in the weak sense of currents, and $\|\varphi\|_{L^\infty(X)}\leq C(X,\omega,\beta, M,\|f\|_p)$.
\end{thm}

 \section{Weak solutions to  the degenerate CMA equations: the case of $\lambda>0$}
 Let $(X,\omega)$ be a compact Hermitian manifold  of complex dimension $n$.  Let $\beta$ be a smooth real $(1,1)$-form on $X$ such that there exists  $\rho\in \mbox{PSH}(X,\beta)\cap L^\infty(X)$. Let $0\leq f\in L^p(X,\omega^n)$, $p>1$, be such that $\int_Xf\omega^n >0$. Let $\lambda>0$ be a positive number.  In this section, we study the weak solutions to the degenerate CMA equations:
\[ (\beta+dd^c\varphi)=e^{\lambda\varphi}f\omega^n, \varphi\in\mbox{PSH}(X,\beta)\cap L^\infty(X). \]

After a rescaling, we only need to consider $\lambda=1$, i.e. the equation 
\[(\beta+dd^c\varphi)=e^{\varphi}f\omega^n, \varphi\in\mbox{PSH}(X,\beta)\cap L^\infty(X). \]
Our main theorem in this section is the following 
\begin{thm}[=Theorem \ref{thm: main 2}]\label{thm: lambda=1}
Let $(X,\omega)$ be a compact Hermitian manifold  of complex dimension $n$.  Let $\beta$ be a smooth real $(1,1)$-form on $X$ such that there exists  $\rho\in \mbox{PSH}(X,\beta)\cap L^\infty(X)$. Let $0\leq f\in L^p(X,\omega^n)$, $p>1$, be such that $\int_Xf\omega^n>0$. Let $\lambda>0$ be a positive number.  Then there is a unique real-valued function $\varphi\in \mbox{PSH}(X,\beta )\cap L^\infty(X)$, satisfying
\[(\beta+dd^c\varphi)^n=e^\varphi f\omega^n\]
in the weak sense of currents, and $\|\varphi\|_{L^\infty(X)}\leq C(X,\omega,\beta, M,\|f\|_p)$.
\end{thm}
The idea of the proof is essentially due to \cite{BEGZ10}. It is an application of the Schauder's fixed point theorem. 

\begin{lem}\label{lem: schauder}
	Assume $\Psi$ is a continous mapping from a non-empty convex closed subset of a Hausdorff topological vector space into a compact subset of itself,  then $\Psi$ must have fixed point.
\end{lem}

Let $\mathcal{C}$ be the compact convex subset of $L^1(X,\omega^n)$, consisting of all $\beta$-psh functions $v$ normalized by $\int_Xv\omega^n=0$. Note that such $v$ is uniformly bounded from above.
Define the mapping $\Psi: \mathcal{C} \rightarrow  \mathcal{C}$ as follows. For any $v\in \mathcal C$, by Theorem \ref{thm: main 1-2}, there is a unique $\Psi(v)\in \mathcal C$ such that 
\[(\beta +dd^c\Psi(v))^n=e^{v+c_v}f\omega^n, \sup_X\Psi(v)=0,\]
where 
\[c_v=\log\frac{\int_X \beta^n}{\int_Xe^vf\omega^n}.\]
\begin{lem}\label{lem: est cv}
	There is  uniform constant $C>0$ such that $c_v\leq C$, for any $v\in \mathcal C$.
	\end{lem}
\begin{proof}Now we are going to get a uniform upper bound for $c_v$.
There is a constant $A>0$, such that $\beta<A\omega$. For  any $v\in \mathcal C$,  we have $v\in \mbox{PSH}(X,A\omega)$. Then by \cite[Corollary 1.3]{Ngu16}, there exists a uniform constant $C=C(X,\omega,\beta,\|f\|_p)>0$, such that 
\begin{align*}
\int_X-vf\omega^n\leq C.
\end{align*}
That is
$$\int_Xvf\omega^n \geq-C.$$
Then by the convexity of the exponential function,  we have 
\[\int_Xe^vf\omega^n\geq e^{\int_Xvf\omega^n/\int_Xf\omega^n}\cdot \int_Xf\omega^n\geq e^{-C'(X,\omega,\beta,\|f\|_p)}.\]
Thus 
\[c_v=\log\frac{\int_X \beta^n}{\int_Xe^vf\omega^n}\leq C'.\]
\end{proof}
Hence  for any $v\in \mathcal{C}$,
\[||e^{v+c_v}f\|_p\leq e^{C'}||f||_p\]
is   uniformly bounded.
Then by Theorem \ref{thm: main 1},  we know that $\Psi(v)$ is uniformly bounded by a uniform constant $C(X,\omega,\beta,||f||_p)$. 

To prove Theorem \ref{thm: lambda=1}, by Lemma \ref{lem: schauder}, it suffices to prove the following
\begin{lem}\label{lem: psi continous}
	The mapping $\Psi$ is continous.
\end{lem}
\begin{proof}
	We only need to prove for any $v_j \rightarrow v$ in $L^1(X,\omega)\cap \mathcal C$, then any accumulative point $u$ of $\Psi(v_j)$ must be $\Psi(v)$.
	After subtracting a subsequence, we may assume $\Psi(v_j)\rightarrow u$ in $L^1(X,\omega)$ and almost everywhere (in Lebesgue measure). By dominated convergenve theorem, we have $c_{v_j}\rightarrow c_v$. We also assume that $v_j\rightarrow v$ almost everywhere on $X$ (in Lebesgue measure). 

If we can prove  $u$ satisfies the CMA equation:
\[(\beta+dd^cu)^n=e^{v+c_v}f\omega^n, u\in \mbox{PSH}(X,\beta)\cap L^\infty(X)\]
then by uniqueness in Theorem \ref{thm: main 1-2},  $u=\Psi(v)$. 

By Theorem \ref{thm: main 1-2}, $\Psi(v_j)$ is uniformly bounded, then  $u$ is bounded.

	We denote $\Phi_j:=( \sup_{k\geq j}\Psi(v_k))^*$, then we have $\Phi_j$ decreases to $u$ on X. By \cite[Proposition 2.8]{BT76},
	$$(\beta +dd^c\max\{ \Psi(v_j),...,\Psi(v_k) \})^n \geq \min\{e^{v_j+c_{v_j}}f\omega^n,...,e^{v_k+c_{v_k}}f\omega^n\}.$$
	Since  $\{ \max\{\Psi(v_j),...,\Psi(v_k)\}\}$ is a sequence in $\mbox{PSH}(X,\beta)\cap L^\infty(X)$ which increases to $\varphi_j$ almost everywhere  on X (in Lebesgue measure), then by \cite[Proposition 5.2]{BT82},  
\[(\beta+dd^c\rho+dd^c\max\{ \Psi(v_j),...,\Psi(v_k) \})^n \xrightarrow{k\rightarrow \infty}(\beta+dd^c\Phi_j)^n \geq \inf_{k\geq j}\{e^{v_k+c_{v_k}}f\omega^n\}\]
weakly as measures on $X$.
Since  $c_{v_j}\rightarrow c_v$ and $v_j \rightarrow v$ almost everywhere  on $X$, we get that 
	$$\lim_{j\rightarrow \infty}\inf_{k\geq j}\{e^{v_k+c_{v_k}}f\omega^n\}=\liminf_{j\rightarrow \infty}e^{v_j+c_{v_j}}f\omega^n=e^{v+c_v}f\omega^n \mbox{ almost everywhere}.$$
The above limit also holds as weak convergence of Radon measures.
By \cite[Theorem 2.1]{BT82},  
	$$\lim_{j\rightarrow \infty}(\beta+dd^c\Phi_j)^n=(\beta +dd^c u)^n.$$
Thus  $(\beta +dd^c u)^n \geq e^{v+c_v}f\omega^n.$ By dominated convergence theorem,  $\int_Xe^{v+c_v}f\omega^n=1$. Noting that  both positive measures $(\beta+dd^cu)^n$ and $e^{v+c_v}f\omega^n$ have the same total mass,  we conclude that $(\beta +dd^c u)^n= e^{v+c_v}f\omega^n.$
Thus the proof of Lemma \ref{lem: psi continous} is complete.

\end{proof}

To complete the proof of Theorem \ref{thm: lambda=1}, we are only left  to prove the uniqueness of the solutions.
\begin{thm}\label{thm: uniq of lambda>0}
Let  $0\leq f\in L^p(X,\omega^n)$, $p\textgreater1$ be such that  $\int_Xf\omega^n\textgreater0$. Then there exists at most one bounded solution to the following complex Monge-Amp\`ere equation:
	$$(\beta+dd^c\varphi)^n=e^{\varphi}f\omega^n.$$
\end{thm}

\begin{proof}
Let $\varphi_1$ and $\varphi_2$ be such solutions, and denote by $T_1$ (resp. $T_2$) the closed positive current $\beta+dd^c\varphi_1$ (resp. $\beta+dd^c\varphi_2$).
	Set $d\mu=f\omega^n$. By  the   comparision principle (see Proposition \ref{prop: np comparison principle}), 
	$$\int_{\{\varphi_1\textless\varphi_2-\varepsilon\}}(\beta+dd^c\varphi_2)^n\leq \int_{\{\varphi_1\textless\varphi_2-\varepsilon\}}(\beta+dd^c\varphi_1)^n.$$
	This is equivalent to
	$$\int_{\{\varphi_1\textless\varphi_2-\varepsilon\}}e^{\varphi_2}d\mu\leq\int_{\{\varphi_1\textless\varphi_2-\varepsilon\}}e^{\varphi_1}d\mu.$$
	So we get
	$$\int_{\{\varphi_1\textless\varphi_2-\varepsilon\}}e^{\varphi_2}d\mu\leq e^{-\varepsilon}\int_{\{\varphi_1\textless\varphi_2-\varepsilon\}}e^{\varphi_2}d\mu,$$
thus $\mu(\{\varphi_1\textless \varphi_2-\varepsilon\})=0$, since $\varphi_2$ is bounded on $X$. For the same reason, $\mu(\{\varphi_2\textless \varphi_1-\varepsilon\})=0$. Let $\varepsilon\searrow 0$, we get $\mu(\{\varphi_1=\varphi_2\})=\mu(X)$. Then by the same method  as in the proof of  Theorem \ref{thm: uniq of lambda=0}, we finish the proof of Theorem \ref{thm: uniq of lambda>0}.
\end{proof}



\section{Asymptotics  of solutions of  degenerate CMA equations }
Let $(X,\omega)$ be a compact Hermitian manifold of complex dimension $n$, $\{\beta\}\in H^{1,1}(X,\mathbb R)$ be a real $(1,1)$-class with smooth representative $\beta$. Let $\rho$ be a bounded $\beta$-PSH function. Let  $f\geq0$, $f \in L^p(X,\omega^n)$, $\int_X f\omega^n \textgreater 0$. For any $\lambda\geq 0$, we have proved that the following degenerate CMA equations can be solved:
\[(\beta+dd^c \varphi_{\lambda})^n=e^{\lambda \varphi_{\lambda}+M_{\lambda}}f\omega^n, \sup_X \varphi_{\lambda}=0.\]
In this section, we study the asymptotics of the solutions $\varphi_{\lambda}$  as $\lambda\rightarrow \lambda_0$, for some $\lambda_0\geq 0$. We first get a bound for $M_\lambda$.
\begin{lem}\label{prop: bound mlambda}
There is a constant  $C>0$  depending only  on $(X,\omega,||f||_p)$, such that 
	\begin{center}
		$\log\frac{\int_X \beta^n}{\int_X f\omega^n} \leq M_{\lambda} \leq \lambda C+\log\frac{\int_X \beta^n}{\int_X f\omega^n}.$
	\end{center}
	\end{lem}
\begin{proof}
	The LHS of inequality is simply because $\int_X(\beta+dd^c \varphi_{\lambda})^n\leq e^{M_{\lambda}}\int_Xf\omega^n$.
	From the concavity property of $\log$ function, we get 
	$$\log\frac{\int_X\beta^n}{\int_Xf\omega^n}=\log\frac{\int_Xe^{\lambda \varphi_{\lambda}+M_{\lambda}}f\omega^n}{\int_Xf\omega^n}\geq \frac{\int_X(\lambda\varphi_{\lambda}+M_{\lambda})f\omega^n }{\int_Xf\omega^n}= M_{\lambda}+\lambda \frac{\int_X \varphi_{\lambda}f\omega^n}{\int_X f\omega^n}.$$
	Since $\varphi_{\lambda} \in \mbox{PSH}(X,\beta) \subseteq \mbox{PSH}(X,C\omega)$ and $\sup_X \varphi_{\lambda}=0$, by \cite[Corollary 1.3]{Ngu16}, there exists a uniform constant  $C=C(||f||_p,X,\omega)\textgreater 0$,  such that
	$$\frac{\int_X \varphi_{\lambda}f\omega^n} {\int_Xf\omega^n} \geq -C.$$
	Therefore we get $$M_{\lambda}\leq \lambda C+\log\frac{\int_X \beta^n}{\int_X f\omega^n}.$$
	\end{proof}

\begin{lem}\label{lem: evo sup cap est}
	Assume $\varphi,\psi \in \mbox{PSH}(X,\beta)\cap L^{\infty}(X)$, $\varphi,\psi \leq0\leq \rho$. Let $a \textgreater0$ be a number, and $(\beta+dd^c\varphi)^n=f\omega^n$ for some non-negative $f\in L^p(\omega^n)$. Then there exists $C=C(a,||f||_p,||\rho||_{\infty}+||\psi||_{\infty})\textgreater 0$, for $\forall \varepsilon \textgreater 0$, 
	$$ \sup_X(\psi-\varphi)\leq \varepsilon +C[Cap_{\beta}(\{\varphi \textless \psi-\varepsilon\})]^{\frac{a}{n}}.$$
\end{lem}
\begin{proof}

Set $M=||\rho||_{\infty}+||\psi||_{\infty}$. Since  
\[ \{\varphi \textless (1-t)\psi+t\rho-s\} \subseteq \{\varphi \textless \psi-s+tM\}, \]
replacing  $(s,t)$ by $(s+\frac{M}{1+M}t,\frac{t}{1+M})$ in Lemma \ref{lem: bounded solution cap est level}, we get for all $s$ and $0\leq t\leq1$,
	$$t^nCap_{\beta}(\{\varphi \textless \psi-s-t\}) \leq (1+M)^n \int_{\{\varphi \textless \psi-s\}}(\beta+dd^c\varphi)^n= (1+M)^n \int_{\{\varphi \textless \psi-s\}}f\omega^n.$$
By the H\"older's inequality, 
\[\int_{\{\varphi \textless\psi-s\}}(\beta+dd^c\varphi)^n=\int_{\{\varphi \textless\psi-s\}}f\omega^n\leq||f||_p·vol(\{\varphi\textless \psi-s\})^{\frac{1}{q}},\]
where $q=p/(p-1)$ is the conjugate number of $p$.
By Lemma \ref{lem: vol-cap beta},  we have 
\[vol(\{\varphi\textless \psi-s\})\leq C\exp(-\alpha Cap_{\beta}(\{\varphi\textless \psi-s\})^{-\frac{1}{n}})\leq A[Cap_{\beta}(\{\varphi\textless \psi-s\})]^{q(1+a)/n}\]
where  $A>0$ is a constant depending  only   on $\alpha$. Set 
\[B=(1+M)^n\|f\|_pA^{1/q}.\]
We conclude that 
\[tCap_{\beta}(\{\varphi \textless \psi-s-t\})^{\frac{1}{n}}\leq B [Cap_{\beta}(\{\varphi \textless \psi-s\})^\frac{1}{n}]^{1+a}.\]
Set 
\[f(s):=[Cap_{\beta}(\{\varphi-\psi \textless -s-\varepsilon\})]^{\frac{1}{n}}.\]

Then $f$ satisfies 
\[tf(s+t)\leq B[f(s)]^{1+a}, \forall s\in \mathbb{R},0\leq t\leq1.\]

By Lemma \ref{lem: egz func} (see also \cite[Lemma 2.3, Remarks 2.5]{EGZ09}), there exists $S=S(a,B)$, such that $f(s)=0,\forall s\geq S$, and 
 if $f(0)\textless \frac{1}{(2B)^{1/a}}$, $S\leq \frac{2B}{1-2^{-a}}[f(0)]^{a}$.

Therefore $\{\varphi \textless \psi-s-\varepsilon\}$ is empty for $s\geq S=S(a,B)$, i.e. $\sup_X(\psi-\varphi) \leq \varepsilon +S$. 

If $f(0)=Cap_{\beta}(\{\varphi \textless \psi -\varepsilon\})^{\frac{1}{n}} \textless \frac{1}{(2B)^{1/a}}$,
we have that 
 $$S\leq \frac{2B}{1-2^{-a}}[f(0)]^{a}=C[Cap_{\beta}(\{\varphi \textless \psi -\varepsilon \})]^{\frac{a}{n}},$$ 
 where  $C= 2B/(1-2^{-a})$, and hence
 \[ \sup_X(\psi-\varphi)\leq \varepsilon +C[Cap_{\beta}(\{\varphi \textless \psi-\varepsilon\})]^{\frac{a}{n}}.\]

If 
$f(0)=Cap_{\beta}(\{\varphi \textless \psi -\varepsilon\})^{\frac{1}{n}} \geq \frac{1}{(2B)^{1/a}}$, 
i.e. 
$[Cap_{\beta}(\{\varphi \textless \psi -\varepsilon\})]^{\frac{a}{n}}\geq \frac{1}{2B}$,
  by Theorem \ref{thm: main 2 in sect}, we have  
	$$\sup_X(\psi-\varphi) \leq -\inf_X \varphi \leq C'=C'(X,\omega,\beta,||f||_p).$$
Thus we get 
$$\sup_X(\psi-\varphi) \leq \varepsilon +C[Cap_{\beta}(\{\varphi \textless \psi-\varepsilon\})]^{\frac{a}{n}},$$
by taking  $C=2BC'$. The proof of Lemma \ref{lem: evo sup cap est} is complete.
\end{proof}

\begin{prop}\label{prop: stabi}
	Assume $\varphi,\psi \in \mbox{PSH}(X,\beta)\cap L^{\infty}(X)$, $\varphi,\psi \leq0$, $(\beta+dd^c\varphi)^n \leq f\omega^n,(\beta+dd^c \psi)^n \leq g\omega^n$ for some non-negative $f,g\in L^p(\omega^n)$. Then there exists a uniform constant  $C=C(||\varphi||_{\infty},||\psi||_{\infty},||f||_p,||g||_p)\textgreater0$, such that  for all  $0 \textless \gamma \textless \frac{2}{2+nq}$
	$$||\varphi-\psi||_{\infty}\leq C||\varphi-\psi||_2^{\gamma},$$
	where $q=p/(p-1)$ is the conjugate number of $p$.
\end{prop}
\begin{proof}
	The proof is the same as the proof of \cite[Proposition 3.3]{EGZ09}. Fix $\varepsilon>0$ and $a>0$ to be determined later. It follows from Lemma \ref{lem: evo sup cap est} that,
	\[\|\varphi-\psi\|_{L^\infty(X)}\leq \varepsilon+C_1[Cap_\beta(\{|\varphi-\psi|>\varepsilon\})]^{\alpha/n}.\]
	By Lemma \ref{lem: bounded solution cap est level}, we have 
	\[  Cap_\beta(\{|\varphi-\psi|>\varepsilon\}) \leq \frac{C_2}{\varepsilon^{n+2/q}}   \int_X|\varphi-\psi|^{2/q} (f+g)\omega^n,        \]
where $C_2$ is a uniform bound depending only on $\|\varphi\|_{L^\infty(X)}$ and $\|\psi\|_{L^\infty(X)}$.
Applying H\"older's inequaltiy, it follows that 
\[  Cap_\beta (\{|\varphi-\psi|>\varepsilon\}) \leq \frac{C_3\|f+g\|_{L^p}}{\varepsilon^{n+2/q}}[\|\varphi-\psi\|_{L^2}]^{2/q}.             \]
Choose $\varepsilon:=\|\varphi-\psi\|_{L^2}^\gamma$ with $0<\gamma<2/(2+nq)$. Then 
\[  Cap_\beta (\{|\varphi-\psi|>\varepsilon\}) \leq C_4     [\|\varphi-\psi\|_{L^2}]^{2/q-\gamma (n+2/q)}.         \]
We infer 
\[  \|\varphi-\psi\|_{L^\infty(X)}\leq   \|\varphi-\psi\|_{L^2}^\gamma+C_5\|\varphi-\psi\|_{L^2}^{\gamma'}, \mbox{~where~}    \gamma'=\frac{\alpha}{n}[2/q-\gamma (n+2/q)].  \]
Now choose $a>0$ large enough such that $\gamma\leq \gamma'$, we get the desired estimate.
	\end{proof}
\begin{thm}[=Theorem \ref{thm: main-3}]\label{thm: evo solution}Let $(X,\omega)$ be a compact Hermitian manifold of complex dimension $n$. Let $\{\beta\}\in H^{1,1}(X,\mathbb R)$ be a real $(1,1)$-class with smooth representative $\beta$. Assume that there is a bounded $\beta$-PSH function $\rho$. Let $f\geq 0$, $f\in L^p(X,\omega^n)$ such that $\int_Xf\omega^n>0$. Let $\lambda\geq 0$, and $\varphi_\lambda$ be the unique solution to the complex Monge-Amp\` ere equation 
\[	(\beta+dd^c \varphi_{\lambda})^n=e^{\lambda \varphi_{\lambda}+M_{\lambda}}f\omega^n,~ \varphi_\lambda\in \mbox{PSH}(X,\beta),~\sup_X \varphi_{\lambda}=0.	\]
If for some $\lambda_0\geq 0$,  $\|\varphi_\lambda-\varphi_{\lambda_0}\|_1\rightarrow 0$ as $\lambda\rightarrow \lambda_0$, then 
\[M_\lambda\rightarrow M_{\lambda_0}, \mbox{ and } \varphi_\lambda\rightarrow \varphi_{\lambda_0} \mbox{ uniformly on } X \mbox{ as } \lambda\rightarrow \lambda_0.\]
	\end{thm}
\begin{proof}
	Without lose of generality, we may assume that $\lambda$ is uniformly bounded.
From Lemma \ref{prop: bound mlambda}, $M_\lambda$ is uniformly bounded.  Noting $\sup_X\varphi_\lambda=0$, $e^{\lambda \varphi_\lambda+M_\lambda}f\omega^n\leq Cf\omega^n$ for some uniform constant. Then by Theorem \ref{thm: main 2 in sect}, $\varphi_\lambda$ is uniformly bounded. If $\|\varphi_\lambda-\varphi_{\lambda_0}\|_1\rightarrow 0$ as $\lambda\rightarrow \lambda_0$, by Proposition \ref{prop: stabi}, 
$$\|\varphi_\lambda-\varphi_{\lambda_0}\|_{L^\infty(X)}\rightarrow 0, \mbox{ as } \lambda\rightarrow \lambda_0.$$
Applying Bedford-Taylor convergence theorem \cite{BT82}, we can conclude that $M_\lambda\rightarrow M_{\lambda_0}$ as $\lambda\rightarrow \lambda_0$.
	\end{proof}

\section{Applications}
 In this section, we apply our result to get two applications: the first one is a partial answer to the Tosatti-Weinkove conjecture \cite{TW12} and the second one is a partial answer to the  Demailly-P\u aun conjecture \cite{DP04}.
\subsection{The Tosatti-Weinkove  conjecture}  The main purpose of this section, is to give a partial answer to the Tosatti-Weinkove conjecture (Conjecture \ref{conj: tw conj}).

 More precisely, we  relax the condition that  $\beta$ is a smooth semi-positive $(1,1)$-form in Nguyen's result \cite[Theorem 4.1]{Ngu16}, to that there is a bounded $\rho\in \mbox{PSH}(X,\beta)$. 
\begin{thm}[=Theorem \ref{thm: main4}]\label{thm: TW conj}	
Let 	X be a compact n-dimensional complex manifold. Suppose $\{\beta\} \in H^{1,1}(X,\mathbb{R})$, $\beta+dd^c \rho \geq 0$ for bounded $\beta$-\mbox{PSH} function $\rho$, $\int_X \beta^n \textgreater0$. Let $x_1,...,x_N \in X$ be fixed points and let $\tau_1,...,\tau_N$ be positive real numbers so that 
	$$\sum_{i=1}^{N}\tau_i^n < \int_X \beta^n.$$
	Then there exists a $\beta$-\mbox{PSH} function $\varphi$ with logarithmic poles at $x_1,...,x_N$:
	$$\varphi(z)\leq \tau_j \log|z|+O(1),$$
	in a coordinate neighbourhood $(z_1,...,z_n)$ centered at $x_i$, where $|z|^2=|z_1|^2+...+|z_n|^2.$	
\end{thm}
\begin{rem}
The K\"ahler version is proved by Demailly-P\u aun by using Demailly's celebrated mass concentration technique in \cite{Dem93}. The conjecture for $n=2$ and $n=3$ and some partial results for general $n$ (if $X$ is Moishezon and $\{\beta\}$ is a rational class) was obtained by Tosatti-Weinkove. Nguyen \cite{Ngu16} proved the Tosatti-Weinkove conjecture under the assumption that $\beta\geq0$, i.e., $\beta$ is a smooth semi-positive $(1,1)$-form. 
\end{rem}
\begin{proof}
The main idea is to adapt Demailly's celebrated mass concentration technique to the setting of Theorem \ref{thm: TW conj} (cf. \cite{Ngu16}, \cite{TW10}). For the sake of completeness, we repeated here. Let $(z_1,\cdots, z_n)$ be a coordinate chart centered at $x_j$. Let $\chi:\mathbb R\rightarrow \mathbb R$ be a smooth increasing convex function, such that $\chi(t)=t$ for $t\geq 0$, and $\chi(t)=-\frac{1}{2}$ for $t\leq -1$. For $\varepsilon>0$, we define 
$$\gamma_{j,\varepsilon}=dd^c\left(\chi\left(\log\frac{|z|}{\varepsilon}\right)\right)$$
where $|z|^2=|z_1|^2+\cdots+|z_n|^2$. It is clear that  $\gamma_{j,\varepsilon}$ is a closed positive $(1,1)$ form on this coordinate chart, and $\gamma_{j,\varepsilon}=dd^c\log|z|$ outside $|z|<\varepsilon$. Then $\gamma_{j,\varepsilon}^n=0$ outside $|z|<\varepsilon$, thus can be trivially extended to a smooth non-negative $(n,n)$-form on $X$. Since 
$$\int_X\gamma^n_{j,\varepsilon}=1,$$
then $\gamma^n_{j,\varepsilon}\rightarrow \delta_{x_j}$ as $\varepsilon\rightarrow 0$. Set 
$$\delta=\int_X\beta^n-\sum_{j=1}^N\tau_j^n>0.$$
By Theorem \ref{thm: main 1}, we can solve the (degenerate) CMA equation
$$(\beta+dd^c\varphi_\varepsilon)^n=\sum_{j=1}^N\tau_j^n\gamma^n_{j,\varepsilon}+\delta\frac{\omega^n}{\int_X\omega^n},$$
with $\varphi_\varepsilon\in \mbox{PSH}(X,\beta)\cap L^\infty(X)$, $\sup_X\varphi_\varepsilon=0$. 
Since the family $\{ \varphi_\varepsilon:\sup_X\varphi_\varepsilon=0\}$ is compact in $L^1(X,\omega^n)$, then upto choosing  a subsequence, we may assume that $\varphi_\varepsilon$ converges to a $\varphi\in \mbox{PSH}(X,\beta)$ in $L^1(X,\omega^n)$. 
We claim that the function $\varphi$ has desired singularities.

Let $U$ be a neighborhood of $x_j$ and suppose that $\beta=dd^ch$ for some  $h\in C^\infty(\overline U)$ on $\overline U$. Set $v:=\psi_\varepsilon=h+\varphi_\varepsilon$. Since there is a uniform constant $C>0$ such that $v|_{\partial U}\leq C$. Set 
$$u=\tau_j\left(\chi(\log\frac{|z|}{\varepsilon})+\log\varepsilon\right)+C_1,$$
for a large constant $C_1$. Then for $\varepsilon>0$ small enough
\[u|_{\partial U}=\tau_j\log|z|+C_1,~~v|_{\partial U}\leq C,\]
\[(dd^cv)^n= \tau^n_j\gamma_\varepsilon^n +\frac{\delta}{\int_X\omega^n}\omega^n> (dd^cu)^n~~\mbox{on}~~U.\]
Choose $C_1$ sufficiently large, $u> v$ on $\partial U$,  then by the Bedford-Taylor comparison principle \cite[Theorem 4.1]{BT82}, we get that $u\geq v$ on $U$. 
Thus $$\psi_\varepsilon\leq \tau_j\log(|z|+\varepsilon)+C_2\mbox{~~on~~}U.$$
Then $\varphi\leq \tau_j\log|z|+O(1)$ in $U$.
	\end{proof}
\begin{rem} In Theorem \ref{thm: TW conj}, the condition	$$\sum_{i=1}^{N}\tau_i^n < \int_X \beta^n$$ can be weaken to $$\sum_{i=1}^{N}\tau_i^n \leq \int_X \beta^n.$$
	The proof is almost the same, with a slight modification during applying the Bedford-Taylor comparison principle \cite[Theorem 4.1]{BT82}.
	\end{rem}

		\subsection{The Demailly-P\u aun's conjecture}

The main purpose of this section is to give a partial answer to the Demailly-P\u aun's conjecture (Conjecture \ref{conj: dp conj}).
\begin{thm}[=Theorem \ref{thm: main 5}]\label{thm:dp conj}
Let 	$(X,\omega)$ is a compact Hermitian manifold, with  $\omega$ a  pluriclosed Hermitian metric, i.e. $dd^c\omega=0$. Let  $\{\beta\} \in H^{1,1}(X,\mathbb{R})$ be a real $(1,1)$-class with smooth representative $\beta$, such that  $\beta+dd^c \rho \geq 0$ for some bounded $\beta$-psh function $\rho$, and $\int_X\beta^n>0$. Then $\{\beta\}$ contains a K\"ahler current.
\end{thm}
\begin{rem}
	By the  Demailly's  regualrization theorem, we know that the class $\beta$  with bounded $\beta$-psh potential  is nef,  thus Theorem \ref{thm:dp conj} confirms the   Conjecture \ref{conj: dp conj} under the assumption that  there is a bounded $\beta$-psh function $\rho$, and the Hermitian metric $\omega$ is pluriclosed. Moreover, by Chiose's result \cite[Theorem 0.2]{Ch14}, the manifold $X$ turns out to be K\"ahler.
	\end{rem}
\begin{rem}
When $n=2$, and $\beta$ is semipostive, Conjecture \ref{conj: dp conj} is true  by the work of N. Buchdahl \cite{Buc99, Buc00} and Lamari \cite{Lam1,Lam2}. When $n=3$,  Chiose \cite{Ch16} proved that the assumption of the existence of a bounded $\rho\in \mbox{PSH}(X,\beta)$ in Theorem \ref{thm:dp conj} can be weaken to that $\{\beta\}$ is nef. For  $n\geq 4$,   Nguyen proved the Conjecture \ref{conj: dp conj} is true under the assumption that $\beta$ is semi-positive and there is a pluriclosed Hermitian metric. Actually, all the previous results mentioned here and Theorem \ref{thm:dp conj} can also be seen as   partial answers to a conjecture of Boucksom \cite{Bou}, which asks for the existence of a K\"ahler currents in a pseudo-effective real $(1,1)$-class $\{\beta\}\in H^{1,1}(X,\mathbb R)$ on a compact Hermitian manifold $(X,\omega)$, under the assumption that $vol(\{\beta\})>0$ (for definition, see \cite{Bou}). Boucksom \cite{Bou} proved his conjecture when $(X,\omega)$ is assume to be compact K\"ahler. For compact Hermitian manifolds, and general pseudo-effective class, the second author has partial results \cite{Wan16,Wan19}.

\end{rem}
%

The proof of  Theorem \ref{thm:dp conj} is inspired by \cite{Ngu16}. We need to do some preparations.

		\begin{thm}\label{thm: twist cma dem-pau}
			$(X,\omega)$ is a compact Hermitian manifold of dimension $n$. Let $\{\beta\} \in H^{1,1}(X,\mathbb{R})$ be a real $(1,1)$-class with smooth representative $\beta$ such that there is a bounded $\beta$-psh function $\rho$ and  $\int_X \beta^n \textgreater 0$. Let  $0\leq f \in L^p(X, \omega^n), p>1$, be such that  $\int_X f\omega^n>0$. Let $\lambda>0$ be a constant. Then, there exists a   $\varphi_{\lambda}\in \mbox{PSH}(X,\omega+\beta+dd^c\rho)$, with $\rho+\varphi_{\lambda} \in C(X)$, satisfying: 
			$$(\omega+\beta+dd^c\rho+dd^c \varphi_{\lambda})^n=e^{\lambda \varphi_{\lambda}}f\omega^n,$$
	in the weak sense of currents.	
		\end{thm}
		
		\begin{proof}

Let $\rho_j$ be a decreasing sequence of smooth functions such that 
$\omega_j:=\omega+\beta+dd^c\rho_j \textgreater (1-\frac{1}{2j})\omega$. In particular, $\omega_j$ are Hermitian metrics on $X$ for $j\geq 2$.
Choose a sequence $0 \textless f_j \in L^p(X,\omega^n)\cap C^{\infty}(X)$, such that $||f_j-f||_p \rightarrow 0$.  By  \cite{Che87}, we can solve 
			$$(\omega_j+dd^c\varphi_j)^n=e^{\lambda \varphi_j}f_j\omega^n,$$ 
			with 	$\varphi_j\in  C^{\infty}(X)$ satisfying $\omega_j+dd^c\varphi_j>0$.

			%
Set $M_j:=\sup_X\varphi_j$, and $\psi_j=\varphi_j-M_j$.	Then the above equation reads that 
$$(\omega_j+dd^c\psi_j)^n=e^{\lambda \psi_j+\lambda M_j}f_j\omega^n.$$	
\noindent \textbf{Claim.} $M_j$ is uniformly bounded.\\
We first prove that $M_j$ is uniformly bounded from above. Let $G\in C^{\infty}(X)$ be the Gauduchon function such that $e^G\omega^{n-1}$ is $dd^c$-closed. Assume $\omega+\beta\leq C\omega$ for some sufficiently large $C=C(\omega,\beta)$, then 
$$(C\omega+dd^c\rho_j+dd^c\psi_j)^n\geq(\omega_j+dd^c\psi_j)^n=e^{\lambda \psi_j+\lambda M_j}f_j\omega^n.$$ 
By \cite[Lemma 1.9]{Ngu16}, we get
$$(C\omega+dd^c\rho_j+dd^c\psi_j)\wedge (C\omega)^{n-1}\geq e^{\frac{\lambda \psi_j+\lambda M_j}{n}}f_j^{\frac{1}{n}}\omega^n.$$
Therefore 
$$(C\omega+dd^c\rho_j+dd^c\psi_j)\wedge e^G(C\omega)^{n-1}\geq e^{\frac{\lambda \psi_j+\lambda M_j}{n}}f_j^{\frac{1}{n}}e^G\omega^n.$$
Integration by parts, we get 
$$C^n\int_Xe^G\omega^{n}\geq \int_Xe^{\frac{\lambda \psi_j+\lambda M_j}{n}}f_j^{\frac{1}{n}}e^G\omega^n.$$
Set $A_j=\int_Xf^{\frac{1}{n}}_je^G\omega^n$. 
By the convexity of exponential  function we get
$$\frac{1}{A_j}\int_Xe^{\frac{\lambda \psi_j}{n}}f_j^{\frac{1}{n}}e^G\omega^n\geq\exp(\frac{1}{A_j}\int_X\frac{\lambda \psi_j}{n}f_j^{\frac{1}{n}}e^G\omega^n).$$
Thus
$$M_j\leq\frac{1}{\lambda} \left(n\log\frac{\alpha_0C}{A_j}-\frac{1}{A_j}\int_X\frac{\lambda \psi_j}{n}f_j^{\frac{1}{n}}e^G\omega^n\right).$$
Noting that 
$$\int_X f_j^{\frac{1}{n}}e^G\omega^n \rightarrow \int_X f^{\frac{1}{n}}e^G\omega^n>0,$$	
similar arguments as in \cite[Proof of Claim 2.6]{Ngu16} yields  that $A_j$ is uniformly bounded by a constant indenpendent of $j\geq 1$. 	

Set $d\mu=f_j^{\frac{1}{n}}e^G\omega^n$.  
To get an upper bound for $M_j$, it suffices to show that $-\int_X\frac{\lambda \psi_j}{n}d\mu$ is uniformly bounded from above. Set $N_j=\sup_j(\psi_j+\rho_j)$.
Since $\rho_j$ is uniform bounded on X,  it is easy to see that $\inf_X\rho_j\leq N_j\leq \sup_X\rho_j$, thus $N_j$ is uniformly bounded. Thus we only need to prove $\int_X-(\psi_j+\rho_j-N_j)d\mu$ is uniformly bounded from above.

From H\"older's inequality, for any Borel set $E\subset X$, we have 
$$\mu(E)\leq ||f_j^{\frac{1}{n}}e^G\omega^n||_{np}·vol(E)^{\frac{1}{q}},$$ 
where $q$ satisfies $\frac{1}{q}+\frac{1}{np}=1$.						

From \cite[Corollary 2.4]{DK12}, there are constants $C_1=C_1(p,X,\omega)$ and $C_2=C_2(X,\omega)$, for any compact subset $K\subset X$, 	$$vol(K)\leq C_1Cap_{C\omega}(K)\exp(-C_2Cap_{C\omega}(K)^{-\frac{1}{n}}).$$  In particular,  $vol(K)\leq C'Cap_{C\omega}(K)^2$ for some uniform constant $C'(C_1,C_2)$. 
We have the following computation:
\begin{align*}
	\int_X-(\psi_j+\rho_j-N_j)(x)d\mu(x)&=\int_Xd\mu(x)\int_{0}^{-(\psi_j+\rho_j-N_j)(x)}dt\\
	&=\int_Xd\mu(x)\int_{0}^{1}dt+\int_Xd\mu(x)\int_{1}^{-(\psi_j+\rho_j-N_j)(x)}dt\\
	&=A_j+\int_1^{+\infty}\mu(\{\psi_j+\rho_j -N_j\leq -t\})dt\\
	&\leq A_j+ C'\int_1^{+\infty}Cap_{C\omega}(\{ \psi_j+\rho_j -N_j\leq -t \})^2\\
	&\leq A_j+C'\int_1^{+\infty}Cap_{C\omega}(\{ \psi_j+\rho_j -N_j< -t+\varepsilon \})^2\\
	&\leq A_j+C'C\int_1^{+\infty}\frac{1}{(t-\varepsilon)^2}\leq C'',
\end{align*}
where the third equality follows from the Fubini theorem and the third  inequality follows from \cite[Proposition 2.5]{DK12}, and $C''$ is a uniform constant depending only on $(p,X,\omega,\beta)$.

Next we want to get a lower bound of $M_j$. The proof is inspired by \cite[Claim 2.6]{Ngu16}. By \cite{TW10}, for each $j\geq 1$, there exist a unique $u_j\in \mbox{PSH}(X,\omega_2)\cap C(X)$ with  $\sup_Xu_j=0$, and a unique constant $c_j>0$ such that 
$$(\omega_2+dd^cu_j)^n=c_jf_j\omega^n.$$
By \cite[Lemma 5.9]{KN15}, the sequence $\{c_j\}$ is bounded away from $0$, i.e. there exists a uniform constant $c_0$ such that $c_j>c_0$.
Since $\sup_X\psi_j=0$, we get that 				
\begin{align*}
(\omega_2+dd^c({\rho_j+\psi_j-\rho_2}))^n&=\frac{e^{\lambda\psi_j+\lambda M_j}}{c_j}(\omega_2+dd^c(\rho_j+u_j))^n\\
&\leq  \frac{e^{\lambda M_j}}{c_j}(\omega_2+dd^c(\rho_j+u_j))^n,
\end{align*}
Then by \cite[Corollary 2.4]{Ngu16}, we have that $e^{\lambda M_j}\geq c_j>c_0>0$, i.e. $M_j>\frac{1}{\lambda}\log(c_0)$. We thus complete the proof of the \textbf{Claim}.

Now we rewrite the complex Monge-Amp\`ere equation 
	$$(\omega_j+dd^c\varphi_j)^n=e^{\lambda \varphi_j}f_j\omega^n$$ 
	as 
		$$(\omega_2+dd^c(\varphi_j+\rho_j-\rho_2-B_j))^n=e^{\lambda \varphi_j+\lambda B_j}f_j\omega^n,$$ 
		where $B_j=\sup_X(\varphi_j+\rho_j-\rho_2)$. Let $M_\infty:=\sup_j\|\rho_j\|_{L^\infty}<+\infty$. It is easy to see that $M_j-2M_\infty\leq B_j\leq M_j+2M_\infty$.
By \textbf{Claim}, we know that  $B_j$ is uniformly bounded, and thus
			$$e^{\lambda \varphi_j}f_j\omega^n \leq Cf_j\omega^n$$ 
	for some uniform constant $C>0$. This means that $e^{\lambda \varphi_j+\lambda B_j}f_j$ is uniformly bounded in $L^p(X,\omega)$, since $\|f_j-f\|_p\rightarrow 0$. Up to a subsequence, we assume that $B_j\rightarrow B<\infty$ as $j\rightarrow \infty$. Since $\{\varphi_j+\rho_j-\rho_2-B_j\}\in \mbox{PSH}(X,\omega_2)$ and $\sup_X(\varphi_j+\rho_j-\rho_2-B_j)=0$, up to a subsequence, we may assume that $\{\varphi_j+\rho_j-\rho_2-B_j\}$ is Cauchy in $L^1(X,\omega^n)$. Then from \cite[Corollary 5.10]{KN15}, we get that $\varphi_j+\rho_j-\rho_2-B_j$ is Cauchy in PSH$(X,\omega_2)\cap C(X)$, thus $\phi_j+\rho_j-\rho_2-B_j$ converges uniformly to $\phi+\rho-\rho_2-B\in C(X)$ on $X$.
%
By Bedford-Taylor \cite{BT76}, we have 
			$$(\omega_2 +dd^c (\phi_j+\rho_j-\rho_2-B_j))^n \rightarrow (\omega_2+ dd^c(\phi+\rho-\rho_2-B))^n=(\omega+\beta+dd^c(\phi+\rho-B))^n.$$
Meanwhile, it holds that 		$e^{-\lambda \phi_j+\lambda B_j}f_j \omega^n \rightarrow e^{\lambda \phi +\lambda B}f\omega^n$. Replacing $\phi$ by $\phi-B$, we get 
			$$(\omega+\beta+dd^c\rho+dd^c \phi)^n=e^{\lambda \phi}f\omega^n,$$
	in the weak sense of currents, where $\phi \in \mbox{PSH}(X,\omega+\beta+dd^c\rho)$ and $\rho+\phi \in C(X)$.

		\end{proof}
The following lemma is due to Lamari \cite{Lam1}.
		\begin{lem}[\cite{Lam1}]\label{lem: lamari}
			Let $\alpha$ be a smooth real (1,1)-form, There exists a distribution $\psi$ on X such that $\alpha+dd^c\psi\geq0$ if and only if:
			
			$$\int_X\alpha \wedge \gamma^{n-1}\geq0$$
			for any Gauduchon metric $\gamma$ on X.
		\end{lem}
		\begin{proof}[Proof of Theorem \ref{thm:dp conj}]
	The proof follow the ideas of  Chiose \cite{Ch16}, Popovici \cite{Pop}, and \cite{Ngu16}. We argue by contradiction. Suppose to the contrary, that $\{\beta\}$ does not contain any K\"ahler current,  by Lemma \ref{lem: lamari}, there is a sequence of positive numbers  $\delta_j$ decreases to 0, for any $j\in \mathbb{N}^*$ there exists a Gauduchon metric $g_j$ such that
		$$\int_X(\beta+dd^c\rho-\delta_j\omega)\wedge g_j^{n-1}\leq0.$$
		Set $G_j:=g_j^{n-1}$. Then the above inequality is equivalent to
\begin{align}
	\int_X(\beta+dd^c\rho)\wedge G_j\leq \delta_j\int_X\omega \wedge G_j.\label{equ: dp conj 1}
\end{align}
By Theorem \ref{thm: main 1}, we can solve the following degenerate complex Monge-Amp\`ere equations:
		$$(\beta+dd^c\rho+dd^cv_j)^n=c_j\omega\wedge G_j,\ \ v_j\in \mbox{PSH}(X,\beta+dd^c\rho)\cap L^{\infty}(X),\ \ \sup_Xv_j=0.$$
Taking integration both sides, we get 
\begin{align}c_j=\frac{\int_X\beta^n}{\int_X\omega \wedge G_j}\textgreater0.\label{equ: dp conj 2}
	\end{align}
		Set $\beta_j:=\beta+dd^c\rho+dd^cv_j$. By (\ref{equ: dp conj 1}),
\begin{align}
	\int_X\beta_j\wedge G_j=\int_X(\beta+dd^c\rho)\wedge G_j\leq \delta_j \int_X\omega \wedge G_j.\label{equ: dp conj 3}
	\end{align}
We need the following
\begin{lem}\label{lem: dp conj key lem}
$\left(\int_X\beta_j\wedge G_j\right)\cdot\left(\int_X\beta_j^{n-1}\wedge\omega\right)\geq\frac{c_j}{n}\left(\int_X\omega\wedge G_j\right)^2.$
	\end{lem}
The proof of Lemma \ref{lem: dp conj key lem} is much technique, to make the idea of the proof of Theorem \ref{thm:dp conj} more easier to follow, we 	present the proof of  Lemma \ref{lem: dp conj key lem} at the end of this section.

	Combining (\ref{equ: dp conj 2}) , (\ref{equ: dp conj 3}), with Lemma \ref{lem: dp conj key lem}, we have 
$$\frac{1}{n}\int_X\beta^n\leq\delta_j\int_X\beta_j^{n-1}\wedge \omega.$$
Since $dd^c\omega=0$, 
$$\int_X\beta_j^{n-1}\wedge \omega=\int_X(\beta+dd^c\rho)^{n-1}\wedge\omega.$$
Letting $\delta_j\rightarrow 0$, we get a contradiction, since $\int_X\beta^n\textgreater0$. 
We thus complete the proof of Theorem \ref{thm:dp conj}.

		\end{proof}

\begin{lem}\label{lem: dp conj key lem 2}Let $g$ be a Gauduchon metric on $X$, and set $G=g^{n-1}$.
	For $0\textless\varepsilon\textless1$, let $v_{\varepsilon}\in \mbox{PSH}(X,\beta+dd^c\rho)\cap L^{\infty}(X)$ be the unique solution to
	$$(\beta+dd^c\rho+dd^cv_{\varepsilon})^n=e^{\varepsilon v_{\varepsilon}}\omega\wedge G.$$
	Then
	$$\left(\int_X(\beta+dd^c\rho+dd^c v_{\varepsilon})\wedge G\right)·\left(\int_X(\beta+dd^c\rho+dd^cv_{\varepsilon})^{n-1}\wedge \omega\right) \geq\frac{1}{n}\left(\int_Xe^{\frac{\varepsilon v_{\varepsilon}}{2}}\omega\wedge G\right)^2.$$
\end{lem}
\begin{proof}
	By rescaling, say  $({\beta+dd^c\rho},{\omega},g)$ replaced by $(\varepsilon(\beta+dd^c\rho),\varepsilon\omega,\varepsilon g)$, without loss of generaility,  we can assume $\varepsilon=1$ and  write $u=\varepsilon v_{\varepsilon}$. 
	
Then we have the equation  
\begin{align}
	(\beta+dd^c\rho+dd^cu)^n=e^{u}\omega\wedge G.\label{equ: dp conj ine 1}
	\end{align}	
We are going to show that 
			$$\left(\int_X(\beta+dd^c\rho+dd^c u)\wedge G\right)·\left(\int_X(\beta+dd^c\rho+dd^cu)^{n-1}\wedge \omega\right) \geq\frac{1}{n}\left(\int_Xe^{\frac{u}{2}}\omega\wedge G\right)^2.$$
			
			By Theorem \ref{thm: twist cma dem-pau}, for every $s\geq1$, there exists $u_s\in \mbox{PSH}(X,\beta+dd^c\rho+\frac{1}{s}\omega)$, such that
		\begin{align}(\beta+dd^c\rho+\frac{1}{s}\omega+dd^cu_s)^n=e^{u_s}\omega \wedge G,\ \ \rho+u_s\in C(X).\label{equ: dp conj ine 2}\end{align}
%
		By the proof of Theorem \ref{thm: twist cma dem-pau},  we know that $\rho+u_s$ is the uniform limit of $\rho_k+\phi_{k,s}\in C^\infty(X)$ as $k\rightarrow +\infty$:
			$$(\frac{1}{s}\omega+\beta+dd^c\rho_k+dd^c\phi_{k,s})^n=e^{\phi_{k,s}}\omega\wedge G, \ \  \frac{1}{s}\omega+\beta+dd^c(\rho_k+\phi_{k,s})>0.$$
Set 
$$\tau_{k,s}:=\frac{1}{s}\omega+\beta+dd^c\rho_k+dd^c\phi_{k,s}\textgreater0.$$
By \cite[Claim 4.8]{Ngu16}, we have 
				$$\left(\int_X\tau_{k,s}\wedge G\right)·\left(\int_X\tau_{k,s}^{n-1}\wedge \omega\right)\geq\frac{1}{n}\left(\int_X\sqrt{\frac{\tau_{k,s}^n}{\omega^n}·\frac{\omega\wedge G}{\omega^n}}\omega^n\right)^2.$$
Because $\rho_k+\phi_{k,s}$ uniformly convergence to $\rho+u_s$, by Bedford-Taylor \cite{BT76}, we have 
\begin{align*}\tau_{k,s}\wedge G&\rightarrow(\frac{1}{s}\omega+\beta+dd^c\rho+dd^cu_s)\wedge G,\\
			\tau_{k,s}^{n-1}\wedge \omega &\rightarrow (\frac{1}{s}\omega+\beta+dd^c\rho+dd^cu_s)^{n-1}\wedge \omega,
			\end{align*}
weakly as $k\rightarrow \infty$, thus
\begin{align}
	&\left(\int_X(\frac{1}{s}\omega+\beta+dd^c\rho+dd^cu_s)\wedge G\right)\cdot\left(\int_X(\frac{1}{s}\omega+\beta+dd^c\rho+dd^cu_s)^{n-1}\wedge \omega\right)\label{equ: dp conj lem ine 1} \\
&	\geq \lim_{k\rightarrow\infty}\frac{1}{n}(\int_Xe^{\frac{\phi_k}{2}}\omega\wedge G)^2=\frac{1}{n}(\int_Xe^{\frac{u_s}{2}}\omega\wedge G)^2.\notag
	\end{align}
\noindent\textbf{Claim 1.} The sequence $\{u_s\}$ is decreasing  as $s\rightarrow \infty$, and $u_s\geq u$. In particular, $\{u_s\}$ is uniformly bounded.

To see this, let $0<s_1<s_2$ and $s_3>0$ be a sufficiently large integer, such that 
$$\frac{\omega}{s_1}+\beta+dd^c\rho_{s_3}>\frac{\omega}{s_1}+\beta+dd^c\rho_{s_3}>0.$$
From (\ref{equ: dp conj ine 1}) and (\ref{equ: dp conj ine 2}), we have
\begin{align*}
	((\beta+dd^c\rho_{s_3})+dd^c(\rho-\rho_{s_3}+u))^n&=e^{\rho-\rho_{s_3}+u}\cdot(e^{\rho_{s_3}-\rho})\omega\wedge G,\\
((\frac{\omega}{s_1}+\beta+dd^c\rho_{s_3})+dd^c(\rho-\rho_{s_3}+u_{s_1}))^n&=e^{\rho-\rho_{s_3}+u_{s_1}}\cdot(e^{\rho_{s_3}-\rho}\omega\wedge G),\\
((\frac{\omega}{s_2}+\beta+dd^c\rho_{s_3})+dd^c(\rho-\rho_{s_3}+u_{s_2}))^n&=e^{\rho-\rho_{s_3}+u_{s_2}}\cdot(e^{\rho_{s_3}-\rho}\omega\wedge G).
	\end{align*}
By \cite[Lemma 3.6]{Ngu16}, 
we get that $u_{s_1}\geq u_{s_2}\geq u$. The proof of the \textbf{Claim 1} is complete.


Set $\tilde{u}:=\lim_{s\rightarrow\infty}u_s$. By Bedford-Taylor \cite{BT82}, taking $s\rightarrow \infty$, we have 
			$$(\beta+dd^c\rho+dd^c\tilde{u})^n=e^{\tilde{u}}\omega\wedge G.$$
			Then from Theorem \ref{thm: uniq of lambda>0}, 
%
 $\tilde{u}=u$. Applying Bedford-Taylor convergence theorem \cite{BT82}, we get 
			$$\left(\int_X(\beta+dd^c\rho+dd^c u)\wedge G\right)\cdot\left(\int_X(\beta+dd^c\rho+dd^cu)^{n-1}\wedge \omega\right) \geq\frac{1}{n}\left(\int_Xe^{\frac{u}{2}}\omega\wedge G\right)^2.$$
		The proof of Lemma \ref{lem: dp conj key lem 2} is complete.
	\end{proof}
\begin{proof}[Proof of Lemma \ref{lem: dp conj key lem}]
By Bedford-Taylor convergence theorem \cite{BT82}, it suffices to prove that $v_\varepsilon\rightarrow v_j$ uniformly on $X$  and $e^{\varepsilon v_\varepsilon}\rightarrow c_j$ as $\varepsilon\rightarrow 0$.
	The proof is the same as the proof of \cite[Theorem 3.2]{Ngu16} by considering the complex Monge-Ampere equation: 
	$$(\beta+dd^c(\rho+v_\varepsilon))^n=e^{\varepsilon(\rho+v_\varepsilon)}(e^{-\varepsilon\rho}\omega\wedge G_j).$$
	In fact, by Lemma \ref{prop: bound mlambda}, the sequence $\{\varepsilon M_\varepsilon:=\varepsilon\sup_X(\rho+v_\varepsilon)\}_{0\leq \varepsilon\leq 1}$ is unformly bounded, thus the LHS of the above equation is uniformly bounded in $L^p(X,\omega^n)$. Then by Theorem \ref{thm: main 2 in sect}, $\rho+v_\varepsilon-M_\varepsilon$ is uniformly bounded.  Also the sequence  $\{\rho+v_\varepsilon-M_\varepsilon\}_{0\leq \varepsilon\leq 1}$  is compact in $L^1(X,\omega^n)$. Without loss of generality, we may assume that $\varepsilon M_\varepsilon\rightarrow M$ and  $\{\rho+v_\varepsilon-M_\varepsilon\}$ is Cauchy in $L^1(X,\omega^n)$  as $\varepsilon\rightarrow 0$. An application of Proposition \ref{prop: stabi} yields that $\rho+v_\varepsilon-M_\varepsilon\rightarrow  \tilde v $ uniformly on $X$, and $\sup_X\tilde v=0$. By Bedford-Taylor convergence theorem \cite{BT82}, 
	$$(\beta+dd^c(\rho+v_\varepsilon))^n=(\beta+dd^c(\rho+v_\varepsilon-M_\varepsilon))^n\rightarrow (\beta+dd^c\tilde v)^n=e^M\omega\wedge G_j.$$
	By integration, one can see that $e^{\varepsilon v_\varepsilon}\rightarrow e^M=c_j$, and  
	by the uniqueness of the above equation (see Theorem \ref{thm: uniq of lambda=0}),  $\tilde v=v_j$. The proof of Lemma \ref{lem: dp conj key lem} is thus complete.
	\end{proof}


	\end{document}